\documentclass[hidelinks,onefignum,onetabnum]{siamart251216}

\usepackage{babel}
\usepackage[T1]{fontenc}
\usepackage[utf8]{inputenc}
\usepackage{siunitx}
\usepackage[dvipsnames,table]{xcolor}
\usepackage{amsmath}
\usepackage{amsfonts}
\usepackage{amssymb}
\usepackage{mathtools}
\usepackage{booktabs}
\usepackage{multirow}
\usepackage{enumitem}
\usepackage{mathdots}
\usepackage{nicematrix}
\usepackage{tikz}
\setlist[enumerate,1]{label=(\roman*)}

\newcommand{\MyCite}[1]{\cite{#1}}
\newcommand{\MyCitet}[2]{{#1}~\cite{#2}}

\newcommand{\MyQED}{}
\newcommand{\AppendixHeader}{}

\usepackage{orcidlink}

\newcommand{\MacroColor}[1]{#1}
\newcommand{\SATMacroColor}[1]{#1}
\newcommand{\QSATMacroColor}[1]{#1}

\newcommand{\EtAl}{\textit{et al.}}

\newcommand{\NonNegativeIntegers}{\mathbb{Z}_{\ge 0}}
\newcommand{\PositiveIntegers}{\mathbb{Z}_{> 0}}

\newcommand{\ZeroVector}{\mathbf{0}}
\newcommand{\OneVector}{\mathbf{1}}

\DeclareMathOperator*{\Argmin}{argmin}

\newcommand{\OptValFunc}{\MacroColor{\mathcal{V}}}
\newcommand{\OptSolSet}{\MacroColor{\mathcal{S}}}
\newcommand{\FeasSolSet}{\MacroColor{\mathcal{F}}}

\newcommand{\ComplexityClassWrapper}[1]{\MacroColor{\ensuremath{#1}}}
\newcommand{\ComplexityClassText}[1]{\textup{#1}}

\newcommand{\ClassP}{\ComplexityClassWrapper{\ComplexityClassText{P}}}
\newcommand{\ClassNP}{\ComplexityClassWrapper{\ComplexityClassText{NP}}}
\newcommand{\ClassCoNP}{\ComplexityClassWrapper{\ComplexityClassText{coNP}}}
\newcommand{\ClassDP}{\ComplexityClassWrapper{\ComplexityClassText{DP}}}
\newcommand{\ClassPiP}[1]{\ComplexityClassWrapper{\Pi^{p}_{#1}}}
\newcommand{\ClassSigmaP}[1]{\ComplexityClassWrapper{\Sigma^{p}_{#1}}}
\newcommand{\ClassDeltaP}[1]{\ComplexityClassWrapper{\Delta^{p}_{#1}}}

\newcommand{\ClassFDeltaP}[1]{\ComplexityClassWrapper{\ComplexityClassText{F}\Delta^{p}_{#1}}}
\newcommand{\ClassFThetaP}[1]{\ComplexityClassWrapper{\ComplexityClassText{F}\Theta^{p}_{#1}}}

\newcommand{\ComplexityProblemWrapper}[1]{\MacroColor{\ensuremath{#1}}}
\newcommand{\ComplexityProblemText}[1]{\textup{#1}}

\newcommand{\DecisionProblemThreeSAT}{\ComplexityProblemWrapper{\ComplexityProblemText{SAT}}}
\newcommand{\SearchProblemThreeSAT}{\ComplexityProblemWrapper{\ComplexityProblemText{LEXSAT}}}
\newcommand{\DecisionProblemLastBitSAT}{\ComplexityProblemWrapper{\ComplexityProblemText{BITSAT}}}

\newcommand{\DecisionProblemQThreeSAT}[1]{\ComplexityProblemWrapper{#1\ComplexityProblemText{-SAT}}}
\newcommand{\SearchProblemQThreeSAT}[1]{\ComplexityProblemWrapper{#1\ComplexityProblemText{-LEXSAT}}}
\newcommand{\DecisionProblemLastBitQSAT}[1]{\ComplexityProblemWrapper{#1\ComplexityProblemText{-BITSAT}}}

\newcommand{\EqtagDecisionProblemOnKLP}[2]{\ComplexityProblemWrapper{#1\ComplexityProblemText{-#2}}}

\newcommand{\DecisionProblemModifierCompNoLink}[1]{\ensuremath{\overline{#1}_{\neg\text{link}}}}
\newcommand{\DecisionProblemModifierComp}[1]{\ensuremath{\overline{#1}}}

\newcommand{\DecisionProblemModifierCompBoxedPoly}[1]{\ensuremath{\ensuremath{{}^{\text{unary}}}\overline{#1}^{\le 1}}}

\newcommand{\DecisionProblemModifierNoLinkBoxed}[1]{#1\ensuremath{_{\neg\text{link}}^{\le 1}}}
\newcommand{\DecisionProblemModifierBoxed}[1]{#1\ensuremath{^{\le 1}}}
\newcommand{\DecisionProblemModifierNoLink}[1]{#1\ensuremath{_{\neg\text{link}}}}

\newcommand{\DecisionProblemModifierNoLinkBoxedPoly}[1]{\ensuremath{{}^{\text{unary}}}#1\ensuremath{_{\neg\text{link}}^{\le 1}}}
\newcommand{\DecisionProblemModifierBoxedPoly}[1]{\ensuremath{{}^{\text{unary}}}#1\ensuremath{^{\le 1}}}
\newcommand{\DecisionProblemModifierNoLinkPoly}[1]{\ensuremath{{}^{\text{unary}}}#1\ensuremath{_{\neg\text{link}}}}
\newcommand{\DecisionProblemModifierPoly}[1]{\ensuremath{{}^{\text{unary}}}#1}

\newcommand{\DecisionProblemFeas}[1]{\hyperref[LabelDecisionProblemFeas]{\EqtagDecisionProblemOnKLP{#1}{FEAS}}}
\newcommand{\DecisionProblemVal}[1]{\hyperref[LabelDecisionProblemVal]{\EqtagDecisionProblemOnKLP{#1}{VAL}}}
\newcommand{\DecisionProblemUnb}[1]{\hyperref[LabelDecisionProblemUnb]{\EqtagDecisionProblemOnKLP{#1}{UNB}}}
\newcommand{\DecisionProblemVerify}[1]{\hyperref[LabelDecisionProblemVerify]{\EqtagDecisionProblemOnKLP{#1}{VERIFY}}}
\newcommand{\DecisionProblemAttain}[1]{\hyperref[LabelDecisionProblemAttain]{\EqtagDecisionProblemOnKLP{#1}{ATTAIN}}}
\newcommand{\SearchProblemObj}[1]{\hyperref[LabelSearchProblemObj]{\EqtagDecisionProblemOnKLP{#1}{SEARCH}}}

\newcommand{\KLPInstances}[1]{\ensuremath{#1\textup{-LP}}}

\newcommand{\NVariables}{\MacroColor{n}}
\newcommand{\NConstraints}{\MacroColor{m}}

\newcommand{\SequenceIndex}{\MacroColor{\beta}}

\newcommand{\PXSATNVariables}{\SATMacroColor{n}}
\newcommand{\PXSATBooleanFormula}{\SATMacroColor{F}}
\newcommand{\PXSATBFSet}{\SATMacroColor{\mathcal{B}}}
\newcommand{\PXSATLexMaxSolution}{\SATMacroColor{M}}
\newcommand{\PXSATTruthAssignment}{\SATMacroColor{s}}
\newcommand{\PXSATObjVariable}{\SATMacroColor{t}}
\newcommand{\PXSATAuxVarOne}{\SATMacroColor{u}}
\newcommand{\PXSATAuxVarTwo}{\SATMacroColor{v}}
\newcommand{\PXSATOneIndicator}{\SATMacroColor{g}}
\newcommand{\PXSATOneIndicatorExpanded}{\SATMacroColor{h}}
\newcommand{\PXSATCostCoefOne}{\SATMacroColor{c}}
\newcommand{\PXSATCostCoefTwo}{\SATMacroColor{d}}
\newcommand{\PXSATQBFConstraintCoef}{\SATMacroColor{E}}
\newcommand{\PXSATQBFConstraintRHS}{\SATMacroColor{e}}
\newcommand{\PXSATObjective}{\SATMacroColor{1 - \PXSATObjVariable + (\PXSATCostCoefOne^{\PXSATBooleanFormula})^{\top} \PXSATAuxVarTwo}}
\newcommand{\PXSATYesObjVal}{\SATMacroColor{0}}
\newcommand{\PXSATNoObjVal}{\SATMacroColor{1}}

\newcommand{\PXQSATNVariables}{\QSATMacroColor{n}}
\newcommand{\PXQSATBooleanFormula}{\QSATMacroColor{F}}
\newcommand{\PXQSATBoolVariable}{\QSATMacroColor{\mu}}
\newcommand{\PXQSATQBF}{\QSATMacroColor{H}}
\newcommand{\PXQSATQBFSet}[1]{\QSATMacroColor{\mathcal{B}_{#1}}}
\newcommand{\PXQSATQBFLexMaxSolution}{\QSATMacroColor{M}}
\newcommand{\PXQSATTruthAssignment}{\QSATMacroColor{s}}
\newcommand{\PXQSATObjVariable}{\QSATMacroColor{t}}
\newcommand{\PXQSATAuxVarOne}{\QSATMacroColor{p}}
\newcommand{\PXQSATAuxVarPen}{\QSATMacroColor{q}}
\newcommand{\PXQSATAuxVarTwo}{\QSATMacroColor{r}}
\newcommand{\PXQSATCostCoefOne}{\QSATMacroColor{c}}
\newcommand{\PXQSATCostCoefTwo}{\QSATMacroColor{d}}
\newcommand{\PXQSATQBFConstraintCoef}{\QSATMacroColor{E}}
\newcommand{\PXQSATQBFConstraintRHS}{\QSATMacroColor{e}}
\newcommand{\PXQSATX}{\QSATMacroColor{x}}
\newcommand{\PXQSATY}{\QSATMacroColor{y}}
\newcommand{\PXQSATIsTrue}{\QSATMacroColor{is true}}
\newcommand{\PXQSATIsNotTrue}{\QSATMacroColor{is not true}}
\newcommand{\PXQSATYesObjVal}{\QSATMacroColor{0}}
\newcommand{\PXQSATNoObjVal}{\QSATMacroColor{1}}

\newcommand{\EqtagGeneralKLP}[3]{\MacroColor{\ensuremath{\mathrm{LP}_{\ensuremath{#1}}^{\ensuremath{#2}}\ensuremath{#3}}}}
\newcommand{\EqrefGeneralKLP}[3]{\textup{(}\text{\hyperref[LabelGeneralKLP]{$\EqtagGeneralKLP{#1}{#2}{#3}$}\textup{)}}}

\newcommand{\EqtagNewQLPOriginal}[3]{\MacroColor{\ensuremath{\mathrm{ELP}_{#1}^{#2}#3}}}
\newcommand{\EqrefNewQLPOriginal}[3]{\textup{(}\text{\hyperref[eq:QLPOriginal]{$\EqtagNewQLPOriginal{#1}{#2}{#3}$}\textup{)}}}

\newcommand{\EqtagNewQLPSearch}[3]{\MacroColor{\ensuremath{\mathrm{FLP}_{#1}^{#2}#3}}}
\newcommand{\EqrefNewQLPSearch}[3]{\textup{(}\text{\hyperref[eq:QLPSearch]{$\EqtagNewQLPSearch{#1}{#2}{#3}$}\textup{)}}}

\newcommand{\EqtagHardnessOfAttain}[3]{\MacroColor{\ensuremath{\mathrm{GLP}_{\ensuremath{#1}}^{\ensuremath{#2}}\ensuremath{#3}}}}
\newcommand{\EqrefHardnessOfAttain}[3]{\textup{(}\text{\hyperref[eq:LabelHardnessOfAttain]{$\EqtagHardnessOfAttain{#1}{#2}{#3}$}\textup{)}}}

\newcommand{\EqtagHardnessOfAttainDecoupled}[3]{\MacroColor{\ensuremath{\widehat{\mathrm{GLP}}_{\ensuremath{#1}}^{\ensuremath{#2}}\ensuremath{#3}}}}
\newcommand{\EqrefHardnessOfAttainDecoupled}[3]{\textup{(}\text{\hyperref[eq:LabelHardnessOfAttainDecoupled]{$\EqtagHardnessOfAttainDecoupled{#1}{#2}{#3}$}\textup{)}}}

\newcommand{\EqtagHardnessOfAttainabilityLemmaInstance}[3]{\MacroColor{\ensuremath{\mathrm{GLP}_{\ensuremath{#1}}^{\ensuremath{#2}}\ensuremath{#3}}}}
\newcommand{\EqrefHardnessOfAttainabilityLemmaInstance}[3]{\textup{(}\text{\hyperref[eq:HardnessOfAttainabilityLemmaInstance]{$\EqtagHardnessOfAttainabilityLemmaInstance{#1}{#2}{#3}$}\textup{)}}}

\newcommand{\EqtagHardFeas}[3]{\MacroColor{\ensuremath{\mathrm{HLP}_{#1}^{#2}#3}}}
\newcommand{\EqrefHardFeas}[3]{\textup{(}\text{\hyperref[eq:HardFeas]{$\EqtagHardFeas{#1}{#2}{#3}$}\textup{)}}}

\newcommand{\EqtagHardFeasDecoupled}[3]{\MacroColor{\ensuremath{\widehat{\mathrm{HLP}}_{#1}^{#2}#3}}}
\newcommand{\EqrefHardFeasDecoupled}[3]{\textup{(}\text{\hyperref[eq:HardFeasDecoupled]{$\EqtagHardFeasDecoupled{#1}{#2}{#3}$}\textup{)}}}

\newcommand{\EqtagHardFeasDecoupledTwo}[3]{\MacroColor{\ensuremath{\overline{\mathrm{HLP}}_{#1}^{#2}#3}}}
\newcommand{\EqrefHardFeasDecoupledTwo}[3]{\textup{(}\text{\hyperref[eq:HardFeasDecoupledTwo]{$\EqtagHardFeasDecoupledTwo{#1}{#2}{#3}$}\textup{)}}}

\newcommand{\EqtagJeroslowPenaltyBefore}[3]{\MacroColor{\ensuremath{\mathrm{J}_{\ensuremath{#1}}^{\ensuremath{#2}}(\ensuremath{#3})}}}
\newcommand{\EqrefJeroslowPenaltyBefore}[3]{\textup{(}\text{\hyperref[LabelJeroslowPenaltyBefore]{$\EqtagJeroslowPenaltyBefore{#1}{#2}{#3}$}\textup{)}}}

\newcommand{\EqtagJeroslowPenaltyAfter}[3]{\MacroColor{\ensuremath{\overline{\mathrm{J}}_{\ensuremath{#1}}^{\ensuremath{#2}}(\ensuremath{#3})}}}
\newcommand{\EqrefJeroslowPenaltyAfter}[3]{\textup{(}\text{\hyperref[LabelJeroslowPenaltyAfter]{$\EqtagJeroslowPenaltyAfter{#1}{#2}{#3}$}\textup{)}}}

\newcommand{\EqtagQThreeSATReducedToKLPLeadersProblem}[3]{\MacroColor{\ensuremath{\textup{QLP}_{\ensuremath{#1}}^{\ensuremath{#2}}\ensuremath{#3}}}}
\newcommand{\EqtagQThreeSATReducedToKLPFollowersProblem}[3]{\MacroColor{\ensuremath{\textup{QLP}_{\ensuremath{#1}}^{\ensuremath{#2}}\ensuremath{#3}}}}

\newcommand{\EqrefQThreeSATReducedToKLPLeadersProblem}[3]{\textup{(}\text{\hyperref[eq:LabelQThreeSATReducedToKLPLeadersProblem]{$\EqtagQThreeSATReducedToKLPLeadersProblem{#1}{#2}{#3}$}\textup{)}}}
\newcommand{\EqrefQThreeSATReducedToKLPFollowersProblem}[3]{\textup{(}\text{\hyperref[eq:LabelQThreeSATReducedToKLPLeadersProblem]{$\EqtagQThreeSATReducedToKLPFollowersProblem{#1}{#2}{#3}$}\textup{)}}}

\newcommand{\EqtagExampleOne}[3]{\MacroColor{\ensuremath{\mathrm{ALP}_{\ensuremath{#1}}^{\ensuremath{#2}}\ensuremath{#3}}}}
\newcommand{\EqrefExampleOne}[3]{\textup{(}\text{\hyperref[eq:LabelExampleOne]{$\EqtagExampleOne{#1}{#2}{#3}$}\textup{)}}}

\newcommand{\EqtagExampleTwo}[3]{\MacroColor{\ensuremath{\mathrm{BLP}_{\ensuremath{#1}}^{\ensuremath{#2}}\ensuremath{#3}}}}
\newcommand{\EqrefExampleTwo}[3]{\textup{(}\text{\hyperref[LabelExampleTwo]{$\EqtagExampleTwo{#1}{#2}{#3}$}\textup{)}}}

\newcommand{\EqtagSensitivity}[1]{\MacroColor{\ensuremath{\mathrm{W}\ensuremath{#1}}}}
\newcommand{\EqrefSensitivity}[1]{\textup{(}\text{\hyperref[eq:LabelSensitivity]{$\EqtagExampleOne{#1}$}\textup{)}}}

\newcommand{\EqtagSATBLP}[3]{\MacroColor{\ensuremath{\mathrm{CLP}_{#1}^{#2}#3}}}
\newcommand{\EqrefSATBLP}[3]{\textup{(}\text{\hyperref[eq:SATBLP]{$\EqtagSATBLP{#1}{#2}{#3}$}\textup{)}}}

\newcommand{\EqtagLEXSATBLP}[3]{\MacroColor{\ensuremath{\mathrm{DLP}_{#1}^{#2}#3}}}
\newcommand{\EqrefLEXSATBLP}[3]{\textup{(}\text{\hyperref[eq:LEXSATBLP]{$\EqtagLEXSATBLP{#1}{#2}{#3}$}\textup{)}}}

\newcommand{\EqtagCompactLEXSATBLP}[3]{\MacroColor{\ensuremath{\widehat{\mathrm{DLP}}_{#1}^{#2}#3}}}
\newcommand{\EqrefCompactLEXSATBLP}[3]{\textup{(}\text{\hyperref[eq:CompactLEXSATBLP]{$\EqtagCompactLEXSATBLP{#1}{#2}{#3}$}\textup{)}}}

\newcommand{\SensitivityOptValFunc}[1]{V#1}
\newcommand{\SensitivityFeasSolSet}[1]{F#1}
\newcommand{\SensitivityOptSolSet}[1]{S#1}

\newcommand{\SensitivitySetValuedMapExampleMap}{F}
\newcommand{\SensitivitySetValuedMapExampleFunc}{f}
\newcommand{\SensitivityDomain}{\mathrm{dom}}
\newcommand{\SensitivityGraph}{\mathrm{graph}}

\newcommand{\GraphOfBLPLowerLevelPolyhedra}{\MacroColor{U}}

\newcommand{\BLPLowerLevelPolyhedraIndexSet}{\MacroColor{\Omega}}

\newcommand{\BLPPointSizeEstimate}{\MacroColor{\psi_2}}

\newcommand{\TLPPointSizeEstimate}{\MacroColor{\psi_3}}

\newcommand{\KLPValueSizeEstimate}{\MacroColor{\phi_k}}

\newcommand{\SATUNSAT}{\MacroColor{\textsc{SAT-UNSAT}}}

\newcommand{\Eqtag}[4]{\ensuremath{\text{#1}_{#2}^{#3}#4}}
\newcommand{\Eqref}[5]{\hyperref[#1]{\Eqtag{#2}{#3}{#4}{#5}}}

\newcommand{\ConstrainedMLPwoLAccent}[1]{\hat{#1}}

\title{Price of Coupling \\ in Multilevel Linear Programming}

\author{Nagisa Sugishita\thanks{HEC Montréal
  (\email{nagisa.sugishita@hec.ca}).}
\and Margarida Carvalho\thanks{Université de Montréal 
  (\email{carvalho@iro.umontreal.ca}).}}

\date{\today}

\begin{document}

\maketitle

\begin{abstract}
Multilevel programming is the standard mathematical framework for modeling hierarchical decision-making, yet a comprehensive understanding of its computational complexity remains fragmented.
In this paper, we characterize the computational complexity of deciding the existence of feasible and optimal solutions, as well as computing the optimal objective value in multilevel linear programming (LP).
Our analysis considers various combinations of modeling assumptions, including the presence or absence of linking (coupling) constraints and whether all variables are bounded. 

In particular, we show the feasibility problem of $k$-level LP for $k \ge 2$ is $\Sigma^p_{k-1}$-complete in general. When linking constraints and unbounded variables are absent, feasibility is polynomial-time solvable for $k \le 4$ but becomes $\Sigma^p_{k-1}$-complete for $k \ge 5$, indicating a sharp jump in computational complexity between four and five levels assuming the polynomial hierarchy does not collapse. 
Combined with other results, one major implication is that no polynomial-time Turing machine can transform a bilevel LP instance with linking constraints into one without linking constraints while preserving feasibility unless $\ClassP{} = \ClassNP{}$. In contrast, such machines exist for all $k \ge 5$.

We observe similar phenomena with the computational complexity of deciding the existence of an optimal solution in multilevel LP. 
In the bilevel case, feasibility and boundedness fully characterize the existence of an optimal solution, implying that the problem is DP-complete. However, these conditions are insufficient for $k \ge 3$. We show that the problem for $k \ge 3$ is $\ClassDeltaP{k}$-complete. 
Similar to the feasibility problem, the decision of the existence of an optimal solution becomes polynomially solvable for $k=2,3$ without linking constraints and unbounded variables.
However, the problem is $\ClassDeltaP{k}$-complete for $k\ge4$, even with these simplifying assumptions.
Interestingly, the computation of the optimal objective value is $\ClassFDeltaP{k}$-complete for any $k \ge 2$, even without linking constraints and unbounded variables.
We also discuss the extension of our results to the mixed-binary cases.
\end{abstract}

\section{Introduction}

A multilevel programming problem is one in which some decision variables
are constrained to belong to the optimal solution set of another problem.
Following the work of~\MyCitet{Bracken and McGill}{bracken1973mathematical} and~\MyCitet{Candler and Norton}{candler1977multi}, research activity in this area has expanded considerably.
In particular, methodological progress in bilevel programming has produced a substantial literature addressing both theoretical questions and computational techniques; see \MyCitet{Beck \EtAl{}}{beck2026bilevel}, \MyCitet{Carvalho \EtAl{}}{carvalho2025integer}, and \MyCitet{Dempe and Zemkoho}{dempe2020bilevel} for recent overviews.
Interest in trilevel programming has also grown, motivated by its expressive modeling power. Notable applications include the critical node problem~\MyCite{nabli2022complexity}, resource allocation models~\MyCite{cassidy1971efficient}, fortification–interdiction settings~\MyCite{tomasaz2024completeness}, and cyber-security planning~\MyCite{liu2015trilevel}. Further examples of multilevel applications are surveyed in~\MyCite{migdalas2013multilevel,vicente1994bilevel}.

In this work, we focus on multilevel linear programming (LP).
Moreover, we concentrate on the optimistic variant of multilevel models, where, if a lower-level problem possesses multiple optimal solutions, the one that is most advantageous for the (immediately preceding) upper-level decision maker is assumed to be chosen.

In multilevel LP, a constraint is referred to as a linking (coupling) constraint if it involves variables from the followers.
The presence of linking constraints makes multilevel LP more complicated at first glance.
For example, the feasible set of a bilevel LP is connected when there are no linking constraints, but this property does not hold when linking constraints appear in the leader’s problem~\MyCite{Audet_2006,Colson_2007}.
Indeed, one can model any mixed-binary LP as a bilevel LP with linking constraints~\MyCite{audet1997links}.

However, from the perspective of computational complexity, many results hold regardless of the presence of linking constraints. For example, \MyCitet{Jeroslow}{Jeroslow1985} showed that the decision version of $k$-level LP is $\ClassSigmaP{k-1}$-hard, and this remains true even when the instance is restricted to have no linking constraints. Recently, \MyCitet{Sugishita and Carvalho}{Sugishita2026} established its membership in \ClassSigmaP{k-1}, implying that the decision version of $k$-level LP is \ClassSigmaP{k-1}-complete with or without linking constraints. \MyCitet{Rodrigues \EtAl{}}{RodriguesEtAl2024} proved \ClassSigmaP{k-1}-hardness for deciding unboundedness in $k$-level LP with linking constraints, and later \MyCitet{Sugishita and Carvalho}{Sugishita2026} showed that this problem is \ClassSigmaP{k-1}-complete regardless of the presence of linking constraints. Similarly, it follows from \MyCitet{Vicente \EtAl{}}{VicenteEtAl1994} that, for a given bilevel LP and a candidate solution, verifying whether it is a local optimum is \ClassCoNP{}-hard, with or without linking constraints.
We note that all of the above hardness results are based on transformations producing instances with coefficients of polynomial magnitude.
Consequently, under standard complexity assumptions, there cannot exist a pseudo-polynomial-time (oracle) Turing machine for the corresponding decision problems.

There is also growing interest in developing approaches that transform a bilevel LP instance with linking constraints into one without them. One notable example is the work of \MyCitet{Henke \EtAl{}}{henke2025coupling}, which discusses an augmented-Lagrangian-based approach for removing linking constraints. \MyCitet{Sugishita and Carvalho}{sugishita2025complexitybilevellinearprogramming} presented a polynomial-time procedure that transforms a bilevel LP instance with linking constraints into an equivalent one without them, while preserving the optimal solutions. The correctness of these methods relies on certain assumptions, such as the existence of an optimal solution.

In this paper, we analyze the computational complexity of deciding the existence of feasible and optimal solutions in multilevel LP, as well as the computational complexity of computing the optimal objective value, as summarized in Table~\ref{tab:summary}.
These problems are fundamental and of interest per se.
For example, the equivalence of the decision version of LP and its feasibility problem plays a key role in the analysis of the polynomial solvability of LP.
Our results show that such an equivalence fails for bilevel LP without linking constraints.
Furthermore, our results indicate the impacts of linking and unbounded variables on the computational complexity of $k$-level LP, hence the title of this paper.
For instance, deciding the feasibility of a trilevel LP with linking constraints is \ClassSigmaP{2}-complete.
Without linking constraints, the problem becomes \ClassCoNP{}-complete.
Moreover, if there are neither linking constraints nor unbounded variables, the problem is solvable in polynomial time.
We observe that the computational complexity appears to increase with the number of levels~$k$ in multilevel LP. A somewhat surprising result is that, without linking constraints and unbounded variables, feasibility can be decided in polynomial time for $k \le 4$, but becomes \ClassSigmaP{k-1}-complete for $k \ge 5$.
Thus, there is a sharp increase in computational complexity when moving from 4-level to 5-level LPs, assuming the polynomial hierarchy does not collapse.
With linking constraints, the problem is \ClassSigmaP{k-1}-complete for any $k \ge 2$.

An interesting implication of the complexity results for feasibility decisions is that unless $\ClassP{} = \ClassNP{}$ (respectively, $\ClassCoNP{} = \ClassSigmaP{2}$, which in turn implies $\ClassNP{} = \ClassCoNP{}$), there is no polynomial-time Turing machine that transforms a bilevel LP (respectively, trilevel LP) instance with linking constraints into one without linking constraints while preserving feasibility.
In this regard, the assumption of the existence of optimal solutions used in the aforementioned transformation of bilevel LP to remove linking constraints in \cite{sugishita2025complexitybilevellinearprogramming} seems unavoidable: we need an assumption that is not verifiable in polynomial time, or we need to give up on the efficiency of the transformation.
In contrast, for $k \ge 5$, such a polynomial-time Turing machine does exist, although such a transformation may not preserve the optimal objective value, as it is designed only to preserve feasibility.

We also analyze the computational complexity of deciding the existence of an optimal solution in $k$-level LP.
In the bilevel case, an optimal solution exists if and only if the instance is both feasible and not unbounded.
Intuitively, this means that deciding the existence of an optimal solution corresponds to the intersection of the feasibility problem and the complement of the unboundedness problem.
The complexity class \ClassDP{} introduced by \MyCitet{Papadimitriou and Yannakakis}{papadimitriou1982complexity} captures decision problems of this type.
For trilevel LP, the situation appears more complex: feasibility and boundedness are necessary but not sufficient for the existence of an optimal solution. In fact, we show that deciding the existence of an optimal solution in a $k$-level LP with $k \ge 3$ is $\ClassDeltaP{k}$-complete.
We note that \ClassDP{} equals $\ClassP{}_{\parallel}^{\ClassNP{}[2]}$, the second level of the query hierarchy over \ClassNP{} containing decision problems solvable with two parallel queries to the \ClassNP{} oracle, while $\ClassDeltaP{k}$ equals $\ClassP{}^{\ClassSigmaP{k-1}}$, a class of decision problems that can be solved by a polynomial-time Turing machine with oracle access to $\ClassSigmaP{k-1}$ (which can make polynomially many adaptive queries) \cite{wagner1990bounded}.
We also examine the special case where the input instance has neither linking constraints nor unbounded variables. As observed for the feasibility problem, the computational complexity of the problem appears to increase sharply from $k=3$ (polynomial-time solvable) to $k=4$ ($\ClassDeltaP{k}$-complete), assuming the polynomial hierarchy does not collapse.

In this paper, we also show that computing the optimal objective value of a $k$-level LP instance is \ClassFDeltaP{k}-complete for $k \ge 2$, under metric reductions computed by logarithmic-space transducers \cite{krentel1986complexity,papadimitriou1982complexity}.
To the best of our knowledge, the computational complexity of any search problem associated with $k$-level LP has not been investigated in the literature, except for \cite{deng1998complexity}.
In that reference, the author cites an unpublished manuscript and claims that the $k$-level LP problem is \ClassFDeltaP{k}-complete.
However, the search problem discussed there is not clearly defined, and neither a proof nor a detailed discussion is provided.
Consequently, we are unable to compare our results with those in \cite{deng1998complexity}.

\begin{table}
\centering
\caption{Computational complexity of the decision version of $k$-level LP (\DecisionProblemVal{k}), the decision of unboundedness (\DecisionProblemUnb{k}), the decision of the existence of a feasible solution (\DecisionProblemFeas{k}), the decision of the existence of an optimal solution (\DecisionProblemAttain{k}) and the search of the optimal objective value (\SearchProblemObj{k}) under various assumptions on linking constraints and variable bounds. ``Bounded $x$'' refers to instances where all variable bounds are finite, while ``unbounded $x$'' denotes general instances. The results without references are established in this paper and ``-'' indicates trivial cases. All results remain valid even when the input instances are restricted to have coefficients of polynomial magnitude.}
\label{tab:summary}
\footnotesize
\setlength\extrarowheight{0.2em}
\setlength{\tabcolsep}{4pt}
\begin{NiceTabular}{ccccccc}
\CodeBefore
\Body
\toprule
& & & \multicolumn{2}{c}{With Linking Constraints} & \multicolumn{2}{c}{Without Linking Constraints} \\
\cmidrule(lr){4-5}
\cmidrule(lr){6-7}
& \multicolumn{1}{c}{$k$} & & \multicolumn{1}{c}{Unbounded $x$} & \multicolumn{1}{c}{Bounded $x$} & \multicolumn{1}{c}{Unbounded $x$} & \multicolumn{1}{c}{Bounded $x$} \\
\midrule
\DecisionProblemVal{k}
& $\ge 2$ & \cite{Jeroslow1985,Sugishita2026} & \ClassSigmaP{k - 1}-complete & \ClassSigmaP{k - 1}-complete & \ClassSigmaP{k - 1}-complete & \ClassSigmaP{k - 1}-complete \\
\midrule
\DecisionProblemUnb{k}
& $\ge 2$ & \cite{Sugishita2026,RodriguesEtAl2024} & \ClassSigmaP{k - 1}-complete & - & \ClassSigmaP{k - 1}-complete & - \\
\midrule
\multirow[m]{4}{*}{\DecisionProblemFeas{k}}
& 2 && \ClassNP{}-complete & \ClassNP{}-complete & \ClassP{} & \ClassP{} \\
& 3 && \ClassSigmaP{2}-complete & \ClassSigmaP{2}-complete & \ClassCoNP{}-complete & \ClassP{} \\
& 4 && \ClassSigmaP{3}-complete & \ClassSigmaP{3}-complete & \ClassPiP{2}-hard & \ClassP{} \\
& $\ge 5$ && \ClassSigmaP{k - 1}-complete & \ClassSigmaP{k - 1}-complete & \ClassSigmaP{k - 1}-complete & \ClassSigmaP{k - 1}-complete \\
\midrule
\multirow[m]{3}{*}{\DecisionProblemAttain{k}}
& 2 & & DP-complete & \ClassNP{}-complete & \ClassCoNP{}-complete & \ClassP{} \\
& 3 & & $\ClassDeltaP{3}$-complete & $\ClassDeltaP{3}$-complete & \ClassPiP{2}-complete & \ClassP{} \\
& $\ge4$ & & $\ClassDeltaP{k}$-complete & $\ClassDeltaP{k}$-complete & $\ClassDeltaP{k}$-complete & $\ClassDeltaP{k}$-complete \\
\midrule
\SearchProblemObj{k}
& $\ge 2$ & & \ClassFDeltaP{k}-complete & \ClassFDeltaP{k}-complete & \ClassFDeltaP{k}-complete & \ClassFDeltaP{k}-complete \\
\bottomrule
\CodeAfter
\end{NiceTabular}
\end{table}

\textbf{Paper Structure.}
This paper is organized as follows.
Section~\ref{sec:ProblemFormulation} presents the problem definitions.
Section~\ref{sec:sensitivity-of-bilevel-linear-program-without-linking-constraints} provides auxiliary results for sensitivity analysis of bilevel LP without linking constraints. 
These are used in Section~\ref{sec:membership_proofs}, where we prove the membership results of the decision of the existence of feasible and optimal solutions and the computation of the optimal objective value in various complexity classes. 
Section~\ref{sec:hardness_results} establishes the hardness results.
Section~\ref{sec:extensions_and_limitations} provides a brief discussion on the extensions and limitations of our argument, including the computational complexity of $k$-level mixed-binary LP.
Finally, Section~\ref{sec:Conclusions} concludes the paper.

\textbf{Notation.}
We say that a vector, matrix, or optimization instance is \emph{rational} if and only if all its data are rational numbers.
Given an optimization instance $P$, we use $\OptValFunc(P)$, $\FeasSolSet(P)$, and $\OptSolSet(P)$ to denote the optimal objective value, the set of feasible solutions, and the set of optimal solutions, respectively.
We write $\NonNegativeIntegers{} = \{0, 1, \ldots, \}$ and $\PositiveIntegers{} = \{1, 2, \ldots, \}$.
We use $\ZeroVector$ (respectively, $\OneVector$) to denote the all-zeros (respectively, all-ones) vector, with the dimension clear from the context.
We use subscripts to denote components of a vector: for a vector $v$, the $i$-th component is denoted by $v_i$.
To reduce clutter, we sometimes refer to column vectors inline; given $u \in \mathbb{R}^{d_1}$ and $v \in \mathbb{R}^{d_2}$, we use
the notation $(u, v)$ to denote the column vector
$
\begin{pmatrix} u \\ v \end{pmatrix} \in \mathbb{R}^{d_1 + d_2}.
$
We follow the convention on binary encoding used in \cite{Schrijver1998}.

\section{Problem Formulation}
\label{sec:ProblemFormulation}

Let $k \in \PositiveIntegers{}$.
In this paper, we use $k$ to denote the number of levels in multilevel programming.
For each $l = 1, \ldots, k$, let $n_l \in \PositiveIntegers{}$ and $m_l \in \NonNegativeIntegers{}$ denote the number of variables and linear constraints in the $l$-th player's problem, respectively.
For each $l = 1, \ldots, k$ and $i = 1, \ldots, k$, let $A_{l i} \in \mathbb{R}^{\NConstraints{}_l \times \NVariables{}_i}$, $b_l \in \mathbb{R}^{\NConstraints{}_l}$ and $c_{l i} \in \mathbb{R}^{\NVariables{}_i}$.
We collectively write $A = \{ A_{li} : l = 1, \ldots, k, i = 1, \ldots, k \}$, $b = \{ b_{l} : l = 1, \ldots, k \}$, and $c = \{ c_{l i} : l = 1, \ldots, k, i = l, \ldots, k \}$.
A $k$-level LP instance $(A, b, c)$ is defined as follows.
Let the first player choose a decision variable $x_1 \in \mathbb{R}^{n_1}$, followed by the second player choosing $x_2 \in \mathbb{R}^{n_2}$, and so on, until the $k$-th player selects $x_k \in \mathbb{R}^{n_k}$. The objective of the $l$-th player is minimizing $\sum_{i = l}^k c_{l i}^{\top} x_i$ while satisfying $\sum_{i=1}^k A_{l i} x_i \ge b_l$.
Formally, the first player's problem~\eqref{LabelGeneralKLP} is defined as
\begin{gather}
\inf_{x_1, \ldots, x_k}
\left\{
\sum_{i = 1}^k c_{1 i}^{\top} x_i
:
\displaystyle
\sum_{i = 1}^k A_{1 i} x_i \ge b_1,
\begin{pmatrix}
x_2 \\ \vdots \\ x_k
\end{pmatrix}
\in
\mathcal{S}\eqref{LabelGeneralKLPk_1Level}
\right\},
\label{LabelGeneralKLP}
\tag{$\EqtagGeneralKLP{k}{1}{(A, b, c)}$}
\end{gather}
where the second player's problem~\eqref{LabelGeneralKLPk_1Level} is
\begin{gather}
\inf_{x_2, \ldots, x_k}
\left\{ \sum_{i = 2}^{k} c_{2 i}^{\top} x_i :
\displaystyle
\sum_{i = 1}^{k} A_{2 i} x_i \ge b_2,
\begin{pmatrix}
x_3 \\ \vdots \\ x_k
\end{pmatrix}
\in
\mathcal{S}(\EqtagGeneralKLP{k}{3}{(x_1, x_2, A, b, c)})
\right\},
\label{LabelGeneralKLPk_1Level}
\tag{$\EqtagGeneralKLP{k}{2}{(x_1, A, b, c)}$}
\end{gather}
with the $l$-th player's problem $(\EqtagGeneralKLP{k}{l}{(x_1, \ldots, x_{l - 1}, A, b, c)})$, $3 \le l \le k - 1$, defined analogously, down to the $k$-th player's problem~\eqref{eq:1Level}:
\begin{gather}
\inf_{x_k}
\left\{ 
c_{k k}^{\top} x_k 
:
\sum_{i = 1}^k A_{k i} x_i \ge b_k
\right\}.
\label{eq:1Level}
\tag{$\EqtagGeneralKLP{k}{k}{(x_1, \ldots, x_{k - 1}, A, b, c)}$}
\end{gather}
We say it is a rational instance if all entries in $(A, b, c)$ are rational.
We identify the $k$-level LP instance $(A, b, c)$ with the first player's problem~\eqref{LabelGeneralKLP}. Accordingly, the optimal objective value of the $k$-level LP instance $(A, b, c)$ is given by $\OptValFunc\eqref{LabelGeneralKLP}$, and a solution $(x_1, \ldots, x_k)$ is feasible for this instance if and only if $(x_1, \ldots, x_k) \in \FeasSolSet\eqref{LabelGeneralKLP}$.

To capture the computational complexity of $k$-level LP instances precisely, we consider the following assumptions.
\begin{enumerate}[label=(C\arabic*)]
\item
\label{Condition1}
Except for the last player's problem~\eqref{eq:1Level}, there are no linear constraints;
\item
\label{Condition2}
The linear constraints in the last player's problem~\eqref{eq:1Level} include variable bounds $\ZeroVector \le x_i \le \OneVector$ for all $i = 1, \ldots, k$;
\item
\label{Condition3}
All entries in $A$, $b$ and $c$ are integers between $-n$ and $n$, where $n = n_1 + \cdots + n_k$.
\end{enumerate}

In bilevel programming, a constraint in the leader's problem is called a linking (coupling) constraint if it involves the followers' variables.
In this paper, given a $k$-level LP instance and $l = 2, \ldots, k$, we say that a constraint in the $l$-th player's problem is a linking (respectively, non-linking) constraint if and only if it involves (respectively, does not involve) the followers' variables.
For example, if $A_{l j} = 0$ for all $j = l + 1, \ldots, k$, the $l$-th player's problem does not have linking constraints at all.
The following proposition will be useful throughout the paper.
It allows one to ``move'' all non-linking constraints to the last player's problem.
Therefore, given a $k$-level LP instance, if it does not have a linking constraint, one can construct an equivalent $k$-level LP instance satisfying condition~\ref{Condition1}.
The proof is found in \MyCite{Sugishita2026}.

\begin{proposition}
\label{prop:constraint-forwarding}
Let $(A, b, c)$ be a $k$-level LP instance.
For each $l = 1, \ldots, k - 1$, suppose $A$ and $b$ are written as
$$
A_{l i} = \begin{pmatrix}
A_{l i}'
\\
A_{l i}''
\end{pmatrix},
\ \
b_{l} = \begin{pmatrix}
b_{l}'
\\
b_{l}''
\end{pmatrix},
\quad
\forall i = 1, \ldots, k,
$$
such that $A_{l i}' = 0$ for all $i = l + 1, \ldots, k$ (i.e., $\sum_{i=1}^{l} A_{l i}' x_i \ge b_l'$ are non-linking constraints in the $l$-th player's problem).
Define a $k$-level LP instance $(\ConstrainedMLPwoLAccent{A}, \ConstrainedMLPwoLAccent{b}, \ConstrainedMLPwoLAccent{c})$ by $\ConstrainedMLPwoLAccent{A}_{l i} = A_{l i}''$, $\ConstrainedMLPwoLAccent{b}_{l} = b_{l}''$ for all $l = 1, \ldots, k - 1$, $i = 1, \ldots, k$, $\ConstrainedMLPwoLAccent{c} = c$, and
\begin{equation*}
(
\ConstrainedMLPwoLAccent{A}_{k 1}
\
\ConstrainedMLPwoLAccent{A}_{k 2}
\cdots
\ConstrainedMLPwoLAccent{A}_{k k-1}
\
\ConstrainedMLPwoLAccent{A}_{k k}
)
=
\begin{pmatrix}
A_{1 1}' & & & & \\
A_{2 1}' & A_{2 2}' & & & \\
\vdots & & \ddots & & \\
A_{k-1 \, 1}' & A_{k-1 \, 2}' & \cdots & A_{k-1 \, k-1}'\\
A_{k 1} & A_{k 2} & \cdots & A_{k k-1} & A_{k k}
\end{pmatrix},
\ConstrainedMLPwoLAccent{b} = \begin{pmatrix}
b_{1}' \\
b_{2}' \\
\vdots \\
b_{k-1}' \\
b_{k} \\
\end{pmatrix}.
\end{equation*}
Then,
$\FeasSolSet\EqrefGeneralKLP{k}{1}{(\ConstrainedMLPwoLAccent{A}, \ConstrainedMLPwoLAccent{b}, \ConstrainedMLPwoLAccent{c})} = \FeasSolSet\EqrefGeneralKLP{k}{1}{(A, b, c)}$ and $\OptValFunc\EqrefGeneralKLP{k}{1}{(\ConstrainedMLPwoLAccent{A}, \ConstrainedMLPwoLAccent{b}, \ConstrainedMLPwoLAccent{c})} = \OptValFunc\EqrefGeneralKLP{k}{1}{(A, b, c)}$.
\end{proposition}

We use \KLPInstances{k} to designate $k$-level LP.
Furthermore, we use \DecisionProblemModifierNoLink{\KLPInstances{k}}, \DecisionProblemModifierBoxed{\KLPInstances{k}}, and \DecisionProblemModifierPoly{\KLPInstances{k}} to denote $k$-level LP satisfying condition~\ref{Condition1}, \ref{Condition2}, and \ref{Condition3}, respectively.
We may use multiple superscripts and/or subscripts to refer to
 $k$-level LP satisfying multiple conditions.
For example, \DecisionProblemModifierNoLinkBoxedPoly{\KLPInstances{k}} denotes $k$-level LP satisfying conditions~\ref{Condition1}--\ref{Condition3}.

For each $k \ge 2$, we define the decision problems discussed in the introduction as follows.
\begin{align}
&
\parbox{0.85\textwidth}{Given a rational $k$-level LP instance $(A, b, c)$ and a rational number $t$, is there a feasible solution whose objective value is less than or equal to $t$?}
\tag{\EqtagDecisionProblemOnKLP{k}{VAL}}
\label{LabelDecisionProblemVal}
\\[0.8em]
&
\parbox{0.85\textwidth}{Given a rational $k$-level LP instance $(A, b, c)$, is it unbounded?}
\tag{\EqtagDecisionProblemOnKLP{k}{UNB}}
\label{LabelDecisionProblemUnb}
\\
&
\parbox{0.85\textwidth}{Given a rational $k$-level LP instance $(A, b, c)$, does it have a feasible solution?}
\tag{\EqtagDecisionProblemOnKLP{k}{FEAS}}
\label{LabelDecisionProblemFeas}
\\
&
\parbox{0.82\textwidth}{Given a rational $k$-level LP instance $(A, b, c)$, does it have an optimal solution?}
\tag{\EqtagDecisionProblemOnKLP{k}{ATTAIN}}
\label{LabelDecisionProblemAttain}
\end{align}
In this paper, we also consider the search problem to compute the optimal objective value:
\begin{align}
&
\parbox{0.85\textwidth}{Given a rational $k$-level LP instance $(A, b, c)$, output the optimal objective value, or return ``no'' if it is not finite.}
\tag{\EqtagDecisionProblemOnKLP{k}{SEARCH}}
\label{LabelSearchProblemObj}
\end{align}
We denote by
\DecisionProblemModifierNoLink{L},
\DecisionProblemModifierBoxed{L}, and
\DecisionProblemModifierPoly{L} the variants of problem L whose inputs are restricted to satisfy conditions~\ref{Condition1}, \ref{Condition2}, and \ref{Condition3}, respectively.
Variants satisfying multiple conditions are defined analogously using the appropriate superscripts and subscripts.
For example, we write
\DecisionProblemModifierNoLinkBoxedPoly{\DecisionProblemVal{k}} to denote the variant of the decision problem
\DecisionProblemVal{k} associated with instances in
\DecisionProblemModifierNoLinkBoxedPoly{\KLPInstances{k}}.

The goal of this paper is to examine the computational complexity of \DecisionProblemFeas{k}, \DecisionProblemAttain{k}, and \SearchProblemObj{k}, establishing the new results summarized in Tables~\ref{tab:summary}.
We conclude this section by restating the prior results of~\cite{Jeroslow1985} and~\cite{Sugishita2026}, which correspond to the first two rows of the table:
\begin{theorem}[\cite{Jeroslow1985, Sugishita2026}]
\mbox{}
\begin{enumerate}
\item
The decision problem \DecisionProblemVal{k} is \ClassSigmaP{k - 1}-complete.
Moreover, this remains valid even when the input is restricted to satisfy conditions~\ref{Condition1}--\ref{Condition3}.
\item
The decision problem \DecisionProblemUnb{k} is \ClassSigmaP{k - 1}-complete.
Moreover, this remains valid even when the input is restricted to satisfy conditions~\ref{Condition1} and~\ref{Condition3}.
\end{enumerate}
\end{theorem}

\section{Sensitivity of Bilevel Linear Program Without Linking Constraints}
\label{sec:sensitivity-of-bilevel-linear-program-without-linking-constraints}

This section contains the sensitivity analysis of the bilevel LP without linking constraints.
The results of this section will be used in later sections to establish results regarding the computational complexity of \KLPInstances{k}.

\MyCitet{Basu \EtAl{}}{BasuEtAl2021} show that a union of polyhedra (e.g., a bounded mixed-integer set) can be modeled as the feasible set of a bilevel LP instance with linking constraints.
As a consequence, the value function of such a bilevel LP instance can be discontinuous (for instance, consider the value function of a mixed-integer program).
In contrast, as we demonstrate below, the value function is continuous in the case without linking constraints.

We begin with a lemma concerning the sensitivity of LPs.

\begin{lemma}
\label{lemma:lp-sensitivity}
Let $A \in \mathbb{R}^{\NConstraints \times \NVariables}$, $c \in \mathbb{R}^{\NVariables}$ and $v_p(b) = \min_{x \in \mathbb{R}^{\NVariables}} \{ c^{\top} x : A x \ge b\}$.
Let $\Delta$ be such that for each nonsingular submatrix $B$ of $A$ all entries of $B^{-1}$ are at most $\Delta$ in absolute value.
Given $b', b'' \in \mathbb{R}^{\NConstraints}$ such that both $v_p(b')$ and $v_p(b'')$ are finite, the following holds:
\begin{enumerate}
\item
$
| v_p(b') - v_p(b'') | \le n \Delta \| c \|_1 \, \| b' - b'' \|_{\infty}.
$
\item
For each optimal solution $x'$ of $v_p(b')$, there exists an optimal solution $x''$ of $v_p(b'')$ such that
$
\| x' - x'' \|_{\infty} \le n \Delta \| b' - b'' \|_{\infty}.
$
\end{enumerate}
\end{lemma}

This is Theorem 10.5 in \MyCite{Schrijver1998}, extended to real-valued $A$, $b$ and $c$.
The proof in \MyCite{Schrijver1998} carries over to Lemma~\ref{lemma:lp-sensitivity} with minor modifications to accommodate real-valued data.

Throughout the remainder of this section, we focus on a bilevel LP instance without linking constraints.
We assume that $A$ and $c$ are fixed collections of \emph{real-valued} matrices and vectors, respectively, while $b$ is a \emph{real-valued} vector subject to perturbation.
Define a bilevel LP parametrized by $b$:
\begin{align}
\min_{x_1, x_2} \left\{ c_{1 1}^{\top} x_{1} + c_{1 2}^{\top} x_{2} :
x_2 \in \Argmin_{x_2'} \left\{
c_{2 2}^{\top} x_{2}'
:
A_{2 1} x_{1} + A_{2 2} x_{2}' \ge b
\right\}
\right\}.
\tag{\EqtagSensitivity{(b)}}
\label{eq:LabelSensitivity}
\end{align}
We write $\SensitivityOptValFunc{(b)} = \OptValFunc\eqref{eq:LabelSensitivity}$, $\SensitivityFeasSolSet{(b)} = \FeasSolSet\eqref{eq:LabelSensitivity}$ and $\SensitivityOptSolSet{(b)} = \OptSolSet\eqref{eq:LabelSensitivity}$.

We begin by showing that $\SensitivityFeasSolSet{}$, the set-valued function mapping $b$ to the feasible set, is continuous in $b$.
Informally, given two sufficiently close vectors $b^{(1)}$ and $b^{(2)}$, if $(x_1, x_2)$ is feasible for~\eqref{eq:LabelSensitivity} with $b = b^{(1)}$, then there exists a feasible point for~\eqref{eq:LabelSensitivity} with $b = b^{(2)}$ that is close to $(x_1, x_2)$, provided~\eqref{eq:LabelSensitivity} with $b = b^{(2)}$ is feasible.

To formalize this argument, we adopt the notation from \MyCite{AubinAndFrankowska2009}.
Given a set-valued map $\SensitivitySetValuedMapExampleMap$ from $\mathbb{R}^{\NConstraints}$ to $\mathbb{R}^{\NVariables}$, the domain and the graph of $\SensitivitySetValuedMapExampleMap$ are given by
$
\SensitivityDomain{}(\SensitivitySetValuedMapExampleMap)
=
\{
b
\in \mathbb{R}^{\NConstraints}
:
\SensitivitySetValuedMapExampleMap(b) \not= \emptyset
\}
$
and
$
\SensitivityGraph{}(\SensitivitySetValuedMapExampleMap)
=
\{
(b, x)
\in \mathbb{R}^{\NConstraints \times \NVariables}
:
x \in \SensitivitySetValuedMapExampleMap(b)
\},
$
respectively.
We say $F$ is Lipschitz continuous if and only if there exists $M$ such that for any $b, b' \in \SensitivityDomain{}(\SensitivitySetValuedMapExampleMap)$,
$
\SensitivitySetValuedMapExampleMap(b)
\subset
\SensitivitySetValuedMapExampleMap(b')
+
M
\| b - b' \|_2 B,
$
where $B \subset \mathbb{R}^{\NVariables}$ is the unit ball.

\begin{lemma}
\label{lemma:lipschitz-continuity-of-blp-feasible-set-function}
\mbox{}
\begin{enumerate}
\item
The domain $\SensitivityDomain{}(\SensitivityFeasSolSet{})$ is closed.
\item
The set-valued function $\SensitivityFeasSolSet{}$ is Lipschitz continuous.
\end{enumerate}
\end{lemma}

\begin{proof}
\begin{enumerate}
\item
Following~\MyCitet{Basu \EtAl{}}{BasuEtAl2021}, we have
$$
\SensitivityGraph{}(\SensitivityFeasSolSet{})
\allowbreak
=
\bigcup_{\omega \in \BLPLowerLevelPolyhedraIndexSet(A, c)}\GraphOfBLPLowerLevelPolyhedra(\omega, A, b, c),
$$
where
$$
\BLPLowerLevelPolyhedraIndexSet(A, c) = \left\{
\omega \in \{0, 1\}^{\NConstraints} :
\exists
\lambda \ge \ZeroVector
\text{ s.t.\ }
A_{2 2}^{\top} \lambda = c_{2 2},
(\OneVector - \omega)^{\top} \lambda = 0
\right\}
$$
and
$$
\GraphOfBLPLowerLevelPolyhedra(\omega, A, b, c)
= \left\{ \begin{pmatrix}b \\ x_1 \\ x_2\end{pmatrix} :
\begin{array}{l}
\displaystyle
A_{2 1} x_1 + A_{2 2} x_2 \ge b, \\
\displaystyle
\omega^{\top} (b - A_{2 1} x_1 - A_{2 2} x_2) \ge 0
\end{array}
\right\}.
$$
Thus,
\begin{align*}
\SensitivityDomain{}(\SensitivityFeasSolSet{})
=
\mathrm{proj}_{b} \left(
\bigcup_{\omega \in \BLPLowerLevelPolyhedraIndexSet(A, c)}
\GraphOfBLPLowerLevelPolyhedra(\omega, A, b, c)
\right)
=
\bigcup_{\omega \in \BLPLowerLevelPolyhedraIndexSet(A, c)} \mathrm{proj}_{b} \left(
\GraphOfBLPLowerLevelPolyhedra(\omega, A, b, c)
\right),
\end{align*}
which is closed.
\item
Let $b^{(1)}, b^{(2)} \in \SensitivityDomain{}(\SensitivityFeasSolSet{})$.
We show
$
\SensitivityFeasSolSet{(b^{(1)})}
\subset
\SensitivityFeasSolSet{(b^{(2)})}
+
M
\| b^{(1)} - b^{(2)} \|_2 B
$
for some constant $M$.
Pick any real point $(x_1^{(1)}, x_2^{(1)}) \in \SensitivityFeasSolSet{(b^{(1)})}$.
In particular, $A_{2 1} x_1^{(1)} + A_{2 2} x_2^{(1)} \ge b^{(1)}$.
Since $A_{2 1} x_1 + A_{2 2} x_2 \ge b^{(2)}$ is feasible, by Lemma~\ref{lemma:lp-sensitivity}, there exists $(\bar{x}_1^{(2)}, \bar{x}_2^{(2)})$ such that $A_{2 1} \bar{x}_1^{(2)} + A_{2 2} \bar{x}_2^{(2)} \ge b^{(2)}$ and $\| (x_1^{(1)}, x_2^{(1)}) - (\bar{x}_1^{(2)}, \bar{x}_2^{(2)}) \|_{\infty} \le M' \| b^{(1)} - b^{(2)} \|_{\infty}$, where where $M'$ is the constant described in Lemma~\ref{lemma:lp-sensitivity}.

Now, observe that $(x_1^{(1)}, x_2^{(1)}) \in \SensitivityFeasSolSet{(b^{(1)})}$ implies that $x_2^{(1)}$ is optimal for the lower-level problem with $b = b^{(1)}$ and $x_1 = x_1^{(1)}$:
\begin{equation}
\label{eq:lemma:lipschitz-continuity-of-blp-feasible-set-function-first-lp}
\min_{x_2} \{
c_{2 2}^{\top} x_2
:
A_{2 1} x_1^{(1)} + A_{2 2} x_2 \ge b^{(1)} \}.
\end{equation}
Next, consider the lower-level problem with $b = b^{(2)}$ and $x_1 = \bar{x}_1^{(2)}$:
\begin{equation}
\label{eq:lemma:lipschitz-continuity-of-blp-feasible-set-function-second-lp}
\min_{x_2} \{
c_{2 2}^{\top} x_2
:
A_{2 1} \bar{x}_1^{(2)} + A_{2 2} x_2 \ge b^{(2)} \}.
\end{equation}
Note that $\bar{x}_2^{(2)}$ is feasible for~\eqref{eq:lemma:lipschitz-continuity-of-blp-feasible-set-function-second-lp}.
Furthermore, \eqref{eq:lemma:lipschitz-continuity-of-blp-feasible-set-function-second-lp} cannot be unbounded, as otherwise it would imply the unboundedness of~\eqref{eq:lemma:lipschitz-continuity-of-blp-feasible-set-function-first-lp}.
Instances~\eqref{eq:lemma:lipschitz-continuity-of-blp-feasible-set-function-first-lp} and~\eqref{eq:lemma:lipschitz-continuity-of-blp-feasible-set-function-second-lp} are LPs that only differ in their RHS: the RHS of \eqref{eq:lemma:lipschitz-continuity-of-blp-feasible-set-function-first-lp} is $b^{(1)} - A_{2 1} x_1^{(1)}$ while that of \eqref{eq:lemma:lipschitz-continuity-of-blp-feasible-set-function-second-lp} is $b^{(2)} - A_{2 1} \bar{x}_1^{(2)}$.
Thus, by Lemma~\ref{lemma:lp-sensitivity} there is an optimal solution $\hat{x}_2^{(2)}$ for~\eqref{eq:lemma:lipschitz-continuity-of-blp-feasible-set-function-second-lp} such that
\begin{align*}
\| \hat{x}_2^{(2)} - x_2^{(1)} \|_\infty
& \le M'' \| (b^{(1)} - A_{2 1} x_1^{(1)}) - (b^{(2)} - A_{2 1} \bar{x}_1^{(2)}) \|_\infty \\
& \le M'' ( \| b^{(1)} - b^{(2)} \|_\infty + \| A_{2 1} \|_\infty \cdot \| x_1^{(1)} - \bar{x}_1^{(2)} \|_\infty ) \\
& \le M'' ( 1 + \| A_{2 1} \|_\infty M' ) \| b^{(1)} - b^{(2)} \|_\infty,
\end{align*}
where $M''$ is the constant given in Lemma~\ref{lemma:lp-sensitivity}.
Now, observe that $(\bar{x}_1^{(2)}, \hat{x}_2^{(2)}) \in \SensitivityFeasSolSet{(b^{(2)})}$.
Using the equivalence of norms $\| \cdot \|_2$ and $\| \cdot \|_\infty$, we have
\begin{align*}
\left\|
\begin{pmatrix}
x_1^{(1)} \\ x_2^{(1)}
\end{pmatrix}
-
\begin{pmatrix}
\bar{x}_1^{(2)} \\ \hat{x}_2^{(2)}
\end{pmatrix}
\right\|_2
\le
M
\| b^{(1)} - b^{(2)} \|_2,
\end{align*}
where $M$ is a constant depending on $M'$ and $M''$.
Since $(x_1^{(1)}, x_2^{(1)})$ was an arbitrary point in $\SensitivityFeasSolSet{(b^{(1)})}$, it follows that
$
\SensitivityFeasSolSet{(b^{(1)})}
\subset
\SensitivityFeasSolSet{(b^{(2)})}
+
M
\| b^{(1)} - b^{(2)} \|_2 B.
$
\MyQED
\end{enumerate}
\end{proof}

Next, we prove the continuity of the value function $\SensitivityOptValFunc{}$.
Given a function $\SensitivitySetValuedMapExampleFunc$ from $\mathbb{R}^{\NConstraints}$ to $\mathbb{R} \cup \{ +\infty, -\infty\}$, we define the effective domain of $\SensitivitySetValuedMapExampleFunc$ as
$
\SensitivityDomain{}(\SensitivitySetValuedMapExampleFunc) = \{ b \in \mathbb{R}^{\NConstraints} : \SensitivitySetValuedMapExampleFunc(b) < \infty \}.
$
Furthermore, we say that a function $\SensitivitySetValuedMapExampleFunc$ is Lipschitz continuous if and only if $f(b) \not= -\infty$ for all $b$ and there exists $N$ such that
$
| \SensitivitySetValuedMapExampleFunc(b)
-
\SensitivitySetValuedMapExampleFunc(b')
|
\le
N \| b - b' \|_2
$
for any $b, b' \in \SensitivityDomain{}(\SensitivitySetValuedMapExampleFunc)$.
The following theorem establishes the continuity of $\SensitivityOptValFunc{}$.
\begin{theorem}
\label{theorem:lipschitz-continuity-of-blp-value-function}
\mbox{}
\begin{enumerate}
\item
The effective domain $\SensitivityDomain{}(\SensitivityOptValFunc{})$ is closed.
\item
If there exists $b$ such that $\SensitivityOptValFunc{(b)} = -\infty$, then for all $b$, $\SensitivityOptValFunc{(b)} = -\infty$ or $\infty$.
\item
Suppose $\SensitivityOptValFunc{(b)} > -\infty$ for all $b$.
Then, $\SensitivityOptValFunc{}$ is Lipschitz continuous.
\end{enumerate}
\end{theorem}

\begin{proof}
Assertion~(i) follows from Lemma~\ref{lemma:lipschitz-continuity-of-blp-feasible-set-function} since $\SensitivityDomain{}(\SensitivityOptValFunc{}) = \SensitivityDomain{}(\SensitivityFeasSolSet{})$.
In the following, we prove assertions~(ii) and~(iii) at once.

We assume $\SensitivityDomain{}(\SensitivityFeasSolSet{}) \ne \emptyset$, for otherwise the statements hold trivially.
Pick any $b^{(1)}$, $b^{(2)} \in \SensitivityDomain{}(\SensitivityFeasSolSet{})$ and $(x^{(1)}_1, x^{(1)}_2) \in \SensitivityFeasSolSet{(b^{(1)})}$.
Then, by Lemma~\ref{lemma:lipschitz-continuity-of-blp-feasible-set-function}, there is  $(x^{(2)}_1, x^{(2)}_2) \in \SensitivityFeasSolSet{(b^{(2)})}$ such that
\begin{equation*}
\left\|
\begin{pmatrix}
x^{(1)}_1 \\ x^{(1)}_2
\end{pmatrix}
-
\begin{pmatrix}
x^{(2)}_1 \\ x^{(2)}_2
\end{pmatrix}
\right\|_2
\le
M \| b^{(1)} - b^{(2)} \|_2.
\end{equation*}
The difference of the objective values of the two points can be bounded as
\begin{align*}
\left|
\sum_{j = 1}^2 c_{1 j}^{\top} x^{(1)}_j -
\sum_{j = 1}^2 c_{1 j}^{\top} x^{(2)}_j
\right|
&
\le
\left\|
\begin{pmatrix}
c_{1 1} \\ c_{1 2}
\end{pmatrix}
\right\|_2
\,
\left\|
\begin{pmatrix}
x^{(1)}_1 \\ x^{(1)}_2
\end{pmatrix}
-
\begin{pmatrix}
x^{(2)}_1 \\ x^{(2)}_2
\end{pmatrix}
\right\|_2
\\
&
\le
\left\|
\begin{pmatrix}
c_{1 1} \\ c_{1 2}
\end{pmatrix}
\right\|_2
M \| b^{(1)} - b^{(2)} \|_2.
\end{align*}
This implies that we have $\SensitivityOptValFunc{(b^{(2)})} \le \SensitivityOptValFunc{(b^{(1)})} + N \| b^{(1)} - b^{(2)} \|_2$ for some constant $N$.
Thus, assertions~(ii) and (iii) follow.
\MyQED
\end{proof}

Finally, we prove that the feasible and optimal solution sets are closed.

\begin{theorem}
\label{theorem:closedness-of-blp-solution-set}
\mbox{}
\begin{enumerate}
\item
The graph of the feasible solution set $\SensitivityGraph{}(\SensitivityFeasSolSet{})$ is closed.
\item
The graph of the optimal solution set $\SensitivityGraph{}(\SensitivityOptSolSet{})$ is closed.
\end{enumerate}
\end{theorem}

\begin{proof}
\begin{enumerate}
\item
By Lemma~\ref{lemma:lipschitz-continuity-of-blp-feasible-set-function}, $\SensitivityFeasSolSet{}$ is upper semicontinuous and $\SensitivityDomain{}(\SensitivityFeasSolSet{})$ is closed.
Furthermore, for any real $b$, $\SensitivityFeasSolSet{(b)}$ is closed.
Invoke Proposition 1.4.8 of~\MyCite{AubinAndFrankowska2009} to see that $\SensitivityGraph{}(\SensitivityFeasSolSet{})$ is closed.
\item
If $\SensitivityDomain{}(\SensitivityOptValFunc{}) = \emptyset$, or if $\SensitivityOptValFunc{(b)} = -\infty$ for some $b \in \SensitivityDomain{}(\SensitivityOptValFunc{})$, $\SensitivityGraph{}(\SensitivityOptSolSet{}) = \emptyset$ and hence closed.
Thus, we assume $\SensitivityDomain{}(\SensitivityOptValFunc{}) \not= \emptyset$ and $\SensitivityOptValFunc{(b)} \not= -\infty$ for all $b \in \SensitivityDomain{}(\SensitivityOptValFunc{})$.

Let $\{ (x_1^{(i)}, x_2^{(i)}, b^{(i)}) : i = 1, 2, \ldots \}$ be a sequence in $\SensitivityGraph{}(\SensitivityOptSolSet{})$ such that $(x_1^{(i)}, x_2^{(i)}, b^{(i)}) \rightarrow (\bar{x}_1, \bar{x}_2, \bar{b})$.
Below, we prove that $(\bar{x}_1, \bar{x}_2, \bar{b}) \in \SensitivityGraph{}(\SensitivityOptSolSet{})$, or equivalently, i) $(\bar{x}_1, \bar{x}_2) \in \SensitivityFeasSolSet{(\bar{b})}$ (it is feasible) and ii) $c_{1 1}^{\top} \bar{x}_1 + c_{1 2}^{\top} \bar{x}_2 = \SensitivityOptValFunc{(\bar{b})}$ (it attains the optimal objective value).

Since $\SensitivityGraph{}(\SensitivityOptSolSet{}) \subset \SensitivityGraph{}(\SensitivityFeasSolSet{})$ and since $\SensitivityGraph{}(\SensitivityFeasSolSet{})$ is closed, we have $(\bar{x}_1, \bar{x}_2, \bar{b}) \in \SensitivityGraph{}(\SensitivityFeasSolSet{})$, implying $(\bar{x}_1, \bar{x}_2) \in \SensitivityFeasSolSet{(\bar{b})}$.

The continuity of $\SensitivityOptValFunc{}$ and the closedness of $\SensitivityDomain{}(\SensitivityOptValFunc{})$ (Theorem~\ref{theorem:lipschitz-continuity-of-blp-value-function}) implies that $\SensitivityOptValFunc{(b^{(i)})} \rightarrow \SensitivityOptValFunc{(\bar{b})}$.
Together with $\SensitivityOptValFunc{(b^{(i)})} = c_{1 1}^{\top} x_1^{(i)} + c_{1 2}^{\top} x_2^{(i)} \rightarrow c_{1 1}^{\top} \bar{x}_1 + c_{1 2}^{\top} \bar{x}_2$, we have $c_{1 1}^{\top} \bar{x}_1 + c_{1 2}^{\top} \bar{x}_2 = \SensitivityOptValFunc{(\bar{b})}$.

Thus, it follows that $(\bar{x}_1, \bar{x}_2) \in \SensitivityOptSolSet{(\bar{b})}$, and therefore $(\bar{x}_1, \bar{x}_2, \bar{b}) \in \SensitivityGraph{}(\SensitivityOptSolSet{})$.
\MyQED
\end{enumerate}
\end{proof}

As a consequence of (ii) in Theorem~\ref{theorem:closedness-of-blp-solution-set}, we obtain an immediate corollary.

\begin{corollary}
\label{cor:ClosednessOfTrilevelLPFeasibleSet}
The feasible set of a \DecisionProblemModifierNoLink{\KLPInstances{3}} instance is closed.
\end{corollary}

\section{Membership Proofs}
\label{sec:membership_proofs}

In this section, we prove the memberships of the decision and search problems as summarized in Table~\ref{tab:summary}.
To that end, we will make use of the following result:
\begin{theorem}[\MyCite{Sugishita2026}]
\label{theorem:klp-feasibility-membership}
For $k \ge 2$, the following statements hold.
\begin{enumerate}
\item
The decision problem \DecisionProblemFeas{k} is in \ClassSigmaP{k - 1}.
\item
Every rational \KLPInstances{k} instance has a rational feasible solution of polynomial encoding size, provided it is feasible.
\item
Every rational \KLPInstances{k} instance has a rational optimal objective value of polynomial encoding size, provided it is finite.
\end{enumerate}
\end{theorem}
Below, we prove tighter upper bounds for \DecisionProblemFeas{k} under various assumptions, as well as other results related to $\DecisionProblemAttain{k}$ and $\SearchProblemObj{k}$.
In Section~\ref{subsec:FeasibilityOfMultilevelLP}, we study the conditions for the feasibility of $k$-level LP.
Section~\ref{subsec:MainAnalysis} provides the main analysis.

\subsection{Feasibility of Multilevel LP for $k \le 4$}
\label{subsec:FeasibilityOfMultilevelLP}

In this section, we study the conditions for the feasibility of $k$-level LP.
Let $(A, b, c)$ be a rational \KLPInstances{2} instance.
By Theorem~\ref{theorem:klp-feasibility-membership}, there is a polynomial $\BLPPointSizeEstimate$ such that $\FeasSolSet\EqrefGeneralKLP{2}{1}{(A, b, c)} \not= \emptyset$ if and only if there exists a rational point $(x_1, x_2) \in \FeasSolSet\EqrefGeneralKLP{2}{1}{(A, b, c)}$ of encoding size at most $\BLPPointSizeEstimate(\sigma)$, where $\sigma$ is the encoding size of the input $(A, b, c)$.
Similarly, let $\TLPPointSizeEstimate$ be a polynomial that bounds the encoding size of a rational feasible solution of a rational feasible $\KLPInstances{3}$ instance.
Using this notation, we show the following result establishing that, under some conditions, the feasibility of a \KLPInstances{k} instance reduces to the existence of an optimal solution of a \KLPInstances{(k - 1)}.

\begin{lemma}
\label{lemma:EquivalentConditionsForFeasibility}
\mbox{}
\begin{enumerate}
\item
Let $(A, b, c)$ be a rational \DecisionProblemModifierNoLink{\KLPInstances{2}} instance and $\sigma$ be its encoding size.
It is feasible if and only if $\OptSolSet\eqref{lemma:EquivalentConditionsForFeasibility:TTwo} \ne \emptyset$, where
\begin{align*}
\tag{\ensuremath{T_2(A, b, c)}}
\label{lemma:EquivalentConditionsForFeasibility:TTwo}
\min_{x_1, x_2}
\left\{
c_{2 2}^{\top} x_2
:
\begin{array}{l}
-2^{\BLPPointSizeEstimate(\sigma)} \cdot \OneVector \le x_1 \le 2^{\BLPPointSizeEstimate(\sigma)} \cdot \OneVector,
\\
A_{2 1} x_1 + A_{2 2} x_2 \ge b_2
\end{array}
\right\}.
\end{align*}
\item
Let $(A, b, c)$ be a rational \DecisionProblemModifierNoLink{\KLPInstances{3}} instance and $\sigma$ be its encoding size.
It is feasible if and only if $\OptSolSet\eqref{lemma:EquivalentConditionsForFeasibility:TThree} \not= \emptyset$, where
\begin{align}
\tag{\ensuremath{T_3(A, b, c)}}
\label{lemma:EquivalentConditionsForFeasibility:TThree}
\min_{x_1, x_2, x_3}
\left\{
c_{2 2}^{\top} x_2
+ c_{2 3}^{\top} x_3
:
\begin{array}{l}
-2^{\TLPPointSizeEstimate(\sigma)} \cdot \OneVector \le x_1 \le 2^{\TLPPointSizeEstimate(\sigma)} \cdot \OneVector,
\\
x_3 \in
\OptSolSet\EqrefGeneralKLP{3}{3}{(x_1, x_2, A, b, c)}
\end{array}
\right\}.
\end{align}
\item
Let $(A, b, c)$ be a rational \DecisionProblemModifierNoLinkBoxed{\KLPInstances{4}} instance.
It is feasible if and only if $\OptSolSet\eqref{lemma:EquivalentConditionsForFeasibility:TFour} \ne \emptyset$, where
\begin{align*}
\tag{\ensuremath{T_4(A, b, c)}}
\label{lemma:EquivalentConditionsForFeasibility:TFour}
\min_{x_1, x_2, x_3, x_4}
\left\{
\sum_{i = 2}^4 c_{2 i}^{\top} x_i
:
\begin{pmatrix}
x_3 \\ x_4
\end{pmatrix}
\in
\OptSolSet\EqrefGeneralKLP{4}{3}{(x_1, x_2, A, b, c)}
\right\}.
\end{align*}
\end{enumerate}
\end{lemma}

\begin{proof}
We prove assertion~(ii).
The other assertions can be shown similarly.

Let $(A, b, c)$ be a rational \DecisionProblemModifierNoLink{\KLPInstances{3}} instance and $\sigma$ be its encoding size.
Since $\OptSolSet\eqref{lemma:EquivalentConditionsForFeasibility:TThree} \subset \FeasSolSet\EqrefGeneralKLP{3}{1}{(A, b, c)}$, $\OptSolSet\eqref{lemma:EquivalentConditionsForFeasibility:TThree} \not= \emptyset$ implies $\FeasSolSet\EqrefGeneralKLP{3}{1}{(A, b, c)} \not= \emptyset$.
Thus, we need only prove the converse.

Let $(x_1', x_2', x_3') \in \FeasSolSet\EqrefGeneralKLP{3}{1}{(A, b, c)}$ be a rational point of encoding size less than or equal to $\TLPPointSizeEstimate(\sigma)$.
In particular, we have $\| x_1' \|_{\infty} \le 2^{\TLPPointSizeEstimate(\sigma)}$.
Thus, $(x_1', x_2', x_3')$ is feasible for~\eqref{lemma:EquivalentConditionsForFeasibility:TThree}.
Therefore, by Corollary~\ref{cor:ClosednessOfTrilevelLPFeasibleSet}, if it is not unbounded, it has an optimal solution (this follows because \KLPInstances{2} is a special case of \KLPInstances{3}, and hence the closedness result established for \KLPInstances{3} applies directly to \KLPInstances{2}).
In the following, we show that it is indeed not unbounded.

Since $(x_1', x_2', x_3') \in \FeasSolSet\EqrefGeneralKLP{3}{1}{(A, b, c)}$, $\OptValFunc\EqrefGeneralKLP{3}{2}{(x_1', A, b, c)}$ is finite.
Thus, from Theorem~\ref{theorem:lipschitz-continuity-of-blp-value-function}, it follows that $\OptValFunc\EqrefGeneralKLP{3}{2}{(x_1, A, b, c)} > -\infty$ for any $x_1 \in \mathbb{R}^{\NVariables}$.
Now, let $V = \{ x_1 \in \mathbb{R}^{\NVariables} : -2^{\TLPPointSizeEstimate(\sigma)} \cdot \OneVector \le x_1 \le 2^{\TLPPointSizeEstimate(\sigma)} \cdot \OneVector\}$ and observe that
\begin{align*}
-\infty
&<
\min_{
x_1 \in V
}
\OptValFunc\EqrefGeneralKLP{3}{2}{(x_1, A, b, c)}
= \OptValFunc\eqref{lemma:EquivalentConditionsForFeasibility:TThree}.
\end{align*}
Thus,~\eqref{lemma:EquivalentConditionsForFeasibility:TThree} is not unbounded and has an optimal solution.
\MyQED
\end{proof}

Note that in Lemma~\ref{lemma:EquivalentConditionsForFeasibility}, we need to add condition~\ref{Condition2} in assertion~(iii) since we do not have a version of Theorem~\ref{theorem:lipschitz-continuity-of-blp-value-function} for $\DecisionProblemModifierNoLink{\KLPInstances{3}}$.

\subsection{Main Analysis}
\label{subsec:MainAnalysis}

We begin with $\DecisionProblemAttain{k}$ and $\SearchProblemObj{k}$.
\begin{theorem}
\mbox{}
\begin{enumerate}
\item
For $k \ge 2$, the search problem \SearchProblemObj{k} is in $\ClassFDeltaP{k}$.
\item
For $k \ge 2$, the decision problem \DecisionProblemAttain{k} is in $\ClassDeltaP{k}$.
\end{enumerate}
\end{theorem}

\begin{proof}
\begin{enumerate}
\item
Fix $k \ge 2$ and let $(A, b, c)$ be a rational $\KLPInstances{k}$ instance.
We construct a polynomial-time Turing machine with oracle access to $\ClassSigmaP{k-1}$ that computes $\OptValFunc\EqrefGeneralKLP{k}{1}{(A, b, c)}$. 
First, query the oracle to test if the instance is unbounded or infeasible.
If it is, output ``no''.
Otherwise, $\OptValFunc\EqrefGeneralKLP{k}{1}{(A, b, c)}$ is a rational number of encoding size bounded by $\KLPValueSizeEstimate(\sigma)$, where $\KLPValueSizeEstimate$ is a polynomial and $\sigma$ is the encoding size of the input (Theorem~\ref{theorem:klp-feasibility-membership}).
Note that we have $l \le \OptValFunc\EqrefGeneralKLP{k}{1}{(A, b, c)} \le u$, where $l = -2^{\KLPValueSizeEstimate(\sigma)}$ and $u = 2^{\KLPValueSizeEstimate(\sigma)}$.
We run a binary search to find $\OptValFunc\EqrefGeneralKLP{k}{1}{(A, b, c)}$.
First, query the oracle if $\OptValFunc\EqrefGeneralKLP{k}{1}{(A, b, c)} \le 0$.
If the oracle returns ``yes'', $-2^{\KLPValueSizeEstimate(\sigma)} \le \OptValFunc\EqrefGeneralKLP{k}{1}{(A, b, c)} \le 0$, so update $u$ to be $0$.
Otherwise, $0 \le \OptValFunc\EqrefGeneralKLP{k}{1}{(A, b, c)} \le 2^{\KLPValueSizeEstimate(\sigma)} $, so update $l$ to be $0$.
We repeatedly query the oracle if $\OptValFunc\EqrefGeneralKLP{k}{1}{(A, b, c)} \le (l + u) / 2$ and update $l$ and $u$ accordingly, until we have $u - l < 1 / 2^{\KLPValueSizeEstimate(\sigma) + 1}$.
There is only one rational number of encoding size at most $\KLPValueSizeEstimate(\sigma)$ in $[l, u]$, for any two distinct rational numbers $a, b$ of encoding size at most $\KLPValueSizeEstimate(\sigma)$ satisfy $|a - b| \ge 1/2^{\KLPValueSizeEstimate(\sigma) + 1}$.
Thus, if the encoding size of $l$ or $u$ is less than or equal to $\KLPValueSizeEstimate(\sigma)$, output that value.
Otherwise, use the continued fraction method~\MyCite{Schrijver1998} to find the unique rational number of encoding size less than or equal to $\KLPValueSizeEstimate(\sigma)$ in the interval $[l, u]$, and output it.
The number of iterations of the binary search is polynomial in $\sigma$, and the above computation runs in polynomial time.
Thus, $\SearchProblemObj{k}$ is in \ClassFDeltaP{k}.

\item
It follows from assertion~(i) and the fact that \DecisionProblemFeas{k} is in \ClassSigmaP{k - 1}.
\end{enumerate}
\end{proof}



Under appropriate assumptions, we obtain significantly tighter upper bounds of the computational complexity of \DecisionProblemAttain{k}.

\begin{theorem}
\label{theorem:membership}
\mbox{}
\begin{enumerate}
\item
The decision problem \DecisionProblemModifierNoLink{\DecisionProblemFeas{2}} is in \ClassP{}.
\item
The decision problem \DecisionProblemAttain{2} is in \ClassDP{}.
\item
The decision problem \DecisionProblemModifierBoxed{\DecisionProblemAttain{2}} is in \ClassNP{}.
\item
The decision problem \DecisionProblemModifierNoLink{\DecisionProblemAttain{2}} is in \ClassCoNP{}.
\item
The decision problem \DecisionProblemModifierNoLinkBoxed{\DecisionProblemAttain{2}} is in \ClassP{}.
\item
The decision problem \DecisionProblemModifierNoLink{\DecisionProblemFeas{3}} is in \ClassCoNP{}.
\item
The decision problem \DecisionProblemModifierNoLinkBoxed{\DecisionProblemFeas{3}} is in \ClassP{}.
\item
The decision problem \DecisionProblemModifierNoLink{\DecisionProblemAttain{3}} is in \ClassPiP{2}.
\item
The decision problem \DecisionProblemModifierNoLinkBoxed{\DecisionProblemAttain{3}} is in \ClassP{}.
\item
The decision problem \DecisionProblemModifierNoLinkBoxed{\DecisionProblemFeas{4}} is in \ClassP{}.
\end{enumerate}
\end{theorem}

\begin{proof}
\begin{enumerate}
\item
Invoke Lemma~\ref{lemma:EquivalentConditionsForFeasibility}.
\item
We have \DecisionProblemAttain{2} $=$ \DecisionProblemFeas{2} $\cap$ \DecisionProblemModifierComp{\DecisionProblemUnb{2}}, which is in $\ClassDP{}$.
\item
We have \DecisionProblemModifierBoxed{\DecisionProblemAttain{2}} $=$ \DecisionProblemModifierBoxed{\DecisionProblemFeas{2}}, which is in $\ClassNP{}$.
\item
We have \DecisionProblemModifierNoLink{\DecisionProblemAttain{2}} $=$ \DecisionProblemModifierNoLink{\DecisionProblemFeas{2}} $\cap$ \DecisionProblemModifierCompNoLink{\DecisionProblemUnb{2}}, which is in $\ClassCoNP{}$, for $\ClassCoNP{}$ is closed under intersection.
\item
As in assertion~(ii), we have \DecisionProblemModifierNoLinkBoxed{\DecisionProblemAttain{2}} $=$ \DecisionProblemModifierNoLinkBoxed{\DecisionProblemFeas{2}}, which is in \ClassP{}.
\item
By Lemma~\ref{lemma:EquivalentConditionsForFeasibility}~(ii), \DecisionProblemModifierNoLink{\DecisionProblemFeas{3}} $\le_l$ \DecisionProblemModifierNoLink{\DecisionProblemAttain{2}}.
Now use assertion~(iv).
\item
Similarly, use the fact that \DecisionProblemModifierNoLinkBoxed{\DecisionProblemFeas{3}} $\le_l$ \DecisionProblemModifierNoLinkBoxed{\DecisionProblemAttain{2}}.
\item
By Corollary~\ref{cor:ClosednessOfTrilevelLPFeasibleSet}, \DecisionProblemModifierNoLink{\DecisionProblemAttain{3}} $=$ \DecisionProblemModifierNoLink{\DecisionProblemFeas{3}} $\cap$ \DecisionProblemModifierCompNoLink{\DecisionProblemUnb{3}}, which is in $\ClassPiP{2}$, for $\ClassPiP{2}$ is closed under intersection.
\item
By Corollary~\ref{cor:ClosednessOfTrilevelLPFeasibleSet},
\DecisionProblemModifierNoLinkBoxed{\DecisionProblemAttain{3}} = \DecisionProblemModifierNoLinkBoxed{\DecisionProblemFeas{3}}, which is in \ClassP{}.
\item
Invoke (iii) in Lemma~\ref{lemma:EquivalentConditionsForFeasibility} and use assertion~(ix).
\MyQED
\end{enumerate}
\end{proof}

\section{Hardness Proofs}
\label{sec:hardness_results}

In this section, we discuss the hardness results shown in Table~\ref{tab:summary}.
Section~\ref{subsec:motivational_example} provides two examples demonstrating the complication of $k$-level LP for $k\ge3$: instances with finite but not attainable optimal objective values.
Section~\ref{subsec:BooleanSatisfiabilityProblemAndItsVariants} summarizes the Boolean satisfiability problem and its variants.
These are used on the hardness results presented in Sections~\ref{sec:HardnessProofs:SearchProblems}, \ref{sec:HardnessProofs:AttainabilityProblems}, and \ref{sec:HardnessProofs:FeasibilityProblems}, where we study \SearchProblemObj{k}, \DecisionProblemAttain{k}, and \DecisionProblemFeas{k}, respectively.

\subsection{Illustrative Examples}\label{subsec:motivational_example}

Below, we present two examples that have finite optimal objective values but no optimal solutions.
Note that for \KLPInstances{2}, an optimal solution always exists if the optimal objective value is finite since the feasible set is closed~\MyCite{BasuEtAl2021}.
By Corollary~\ref{cor:ClosednessOfTrilevelLPFeasibleSet}, the same holds for $\DecisionProblemModifierNoLink{\KLPInstances{3}}$.
However, the following examples show that this is not the case for $\DecisionProblemModifierNoLink{\KLPInstances{4}}$ and \KLPInstances{3}.

First, consider the following rational \DecisionProblemModifierBoxed{\KLPInstances{4}} instance:
\begin{align}
&
\inf_{x_1, x_2, x_3, x_4}
\left\{ x_3 - x_1 : 0 \le x_1 \le 1,
(x_2, x_3, x_4)
\in
\OptSolSet\eqref{eq:LabelExampleOneSecondPlayersProblem}
\right\},
\label{eq:LabelExampleOne}
\tag{\EqtagExampleOne{4}{1}{}}
\\
&
\inf_{x_2, x_3, x_4}
\left\{ -x_3 :
x_2 = 0,
0 \le x_2 \le 1,
(x_3, x_4)
\in
\OptSolSet\eqref{eq:LabelExampleOneThirdPlayersProblem}
\right\},
\label{eq:LabelExampleOneSecondPlayersProblem}
\tag{\EqtagExampleOne{4}{2}{(x_1)}}
\\
&
\inf_{x_3, x_4}
\left\{ x_4 :
0 \le x_3 \le 1,
x_4
\in
\OptSolSet\eqref{eq:LabelExampleOneFourthPlayersProblem}
\right\},
\label{eq:LabelExampleOneThirdPlayersProblem}
\tag{\EqtagExampleOne{4}{3}{(x_1, x_2)}}
\\
&
\inf_{x_4}
\left\{ -x_4 :
x_4 \le x_3,
x_4 \le 2 - x_1 - x_3,
0 \le x_4 \le 1
\right\}.
\label{eq:LabelExampleOneFourthPlayersProblem}
\tag{\EqtagExampleOne{4}{4}{(x_1, x_2, x_3)}}
\end{align}
In light of Proposition~\ref{prop:constraint-forwarding}, one can rewrite the instance as a \DecisionProblemModifierNoLinkBoxed{\KLPInstances{4}} instance.
Observe that $\OptSolSet\eqref{eq:LabelExampleOneFourthPlayersProblem} = \left\{ x_4 : x_4 = \min\{ x_3, 2 - x_1 - x_3 \} \right\}$ for any $x_1$, $x_2$, $x_3 \in [0, 1]$.
For any $x_1$, $x_2 \in [0, 1]$, the third player minimizes $\min\{ x_3, 2 - x_1 - x_3 \}$, which is concave in $x_3$.
Thus, the minimum is attained at $x_3 = 0$ and/or $x_3 = 1$.
Direct computation shows that
\begin{align*}
\OptSolSet\eqref{eq:LabelExampleOneThirdPlayersProblem}
&=
\begin{cases}
\{ (x_3, x_4) : x_3 = x_4 = 0 \}, & \text{if } 0 \le x_1 < 1, \\
\{ (x_3, x_4) : x_3 \in \{0, 1\}, x_4 = 0 \}, & \text{if } x_1 = 1.
\end{cases}
\end{align*}
The second player maximizes $x_3$, thus
\begin{align*}
\OptSolSet\eqref{eq:LabelExampleOneSecondPlayersProblem}
&=
\begin{cases}
\{ (x_2, x_3, x_4) : x_2 = x_3 = x_4 = 0 \}, & \text{if } 0 \le x_1 < 1, \\
\{ (x_2, x_3, x_4) : x_2 = 0, x_3 = 1, x_4 = 0 \}, & \text{if } x_1 = 1.
\end{cases}
\end{align*}
Note that for any $x_1 \in [0, 1]$, $\OptSolSet\eqref{eq:LabelExampleOneSecondPlayersProblem}$ is a singleton and its objective value is
$$
\begin{cases}
-x_1, & 0 \le x_1 < 1, \\
1 - x_1, & x_1 = 1.
\end{cases}
$$
Thus, we have $\OptValFunc\eqref{eq:LabelExampleOne} = -1$, but this value is not attained. 
This shows that a rational \DecisionProblemModifierNoLinkBoxed{\KLPInstances{4}} instance may have a finite, but non-attainable optimal objective value.
While \MyCitet{Jeroslow}{Jeroslow1985} provides a similar example, our instance exhibits simpler behavior of the second player, making the subsequent analysis more convenient.

Next, consider the following \DecisionProblemModifierBoxed{\KLPInstances{3}} instance:
\begin{align}
\inf_{x_1, x_2, x_3}
\left\{x_2 - x_1
:
0 \le x_1 \le 1,
(x_2, x_3) \in \OptSolSet\EqrefExampleTwo{3}{2}{(x_1)}
\right\},
\label{LabelExampleTwo}
\tag{\EqtagExampleTwo{3}{1}{}}
\end{align}
\begin{align}
\inf_{x_2, x_3} \left\{
-x_2
:
x_2 \le x_1, x_3 = 0, 0 \le x_2 \le 1,
x_3 \in \OptSolSet\EqrefExampleTwo{3}{3}{(x_1, x_2)}
\right\},
\label{LabelExampleTwoSecond}
\tag{\EqtagExampleTwo{3}{2}{(x_1)}}
\end{align}
\begin{align}
\inf_{x_3}
\left\{
-x_3 :
x_3 \le x_2, x_3 \le 1 - x_2, 0 \le x_3 \le 1
\right\}.
\label{LabelExampleTwoThird}
\tag{\EqtagExampleTwo{3}{3}{(x_1, x_2)}}
\end{align}

At the optimality of the third player, we have $x_3 = \min\{x_2, 1-x_2\}$.
Thus, the second player's constraint $x_3 = 0$ enforces $x_2 = 0$ or $1$.
Since the second player maximizes $x_2$, we have
\begin{align*}
\OptSolSet\eqref{LabelExampleTwoSecond}
&=
\begin{cases}
\{ (x_2, x_3) : x_2 = x_3 = 0 \}, & \text{if } 0 \le x_1 < 1, \\
\{ (x_2, x_3) : x_2 = 1, x_3 = 0 \}, & \text{if } x_1 = 1.
\end{cases}
\end{align*}
Thus, following the same argument as in \EqrefExampleOne{4}{1}{}, we obtain $\OptValFunc\EqrefExampleTwo{3}{1}{} = -1$ but there is no optimal solution.

\subsection{Boolean Satisfiability Problem and Its Variants}
\label{subsec:BooleanSatisfiabilityProblemAndItsVariants}

In our hardness proofs, we use the Boolean satisfiability problem and several of its variants.
In this section, we summarize these problems.

Arguably, the most well-known problem of this type is the satisfiability problem \DecisionProblemThreeSAT{} for Boolean formulae.
Let $\PXSATBFSet$ be the set of Boolean formulae.
The input of \DecisionProblemThreeSAT{} is a Boolean formula $\PXSATBooleanFormula{} \in \PXSATBFSet{}$ and it is a ``yes'' instance if and only if $\PXSATBooleanFormula$ is satisfiable.
In this paper, we identify truth assignments to the Boolean variables with binary vectors, and we say that a binary vector $\PXSATTruthAssignment$ satisfies $\PXSATBooleanFormula$ if and only if the corresponding truth assignment satisfies $\PXSATBooleanFormula$.
For any $\PXSATBooleanFormula{} \in \PXSATBFSet{}$ that is satisfiable, let $\PXSATLexMaxSolution(\PXSATBooleanFormula)$ be the lexicographically maximum binary vector satisfying $\PXSATBooleanFormula$.
This quantity is undefined if $\PXSATBooleanFormula$ is not satisfiable.
The search problem \SearchProblemThreeSAT{} takes $\PXSATBooleanFormula$ as input and requires outputting $\PXSATLexMaxSolution(\PXSATBooleanFormula)$ if $\PXSATBooleanFormula$ is satisfiable, or ``no'' otherwise.
The corresponding decision problem is \DecisionProblemLastBitSAT{}: the input is $\PXSATBooleanFormula$, and it is a ``yes'' instance if and only if $\PXSATBooleanFormula$ is satisfiable and the last (least significant) component of $\PXSATLexMaxSolution(\PXSATBooleanFormula)$ is $1$.
It is known that \DecisionProblemThreeSAT{} is \ClassNP{}-complete, \SearchProblemThreeSAT{} is \ClassFDeltaP{2}-complete, and \DecisionProblemLastBitSAT{} is \ClassDeltaP{2}-complete~\MyCite{krentel1986complexity}.

The above problems can be extended by considering alternating quantifiers.
For $k \in \PositiveIntegers$, let $\PXQSATQBFSet{k}$ denote the set of closed quantified Boolean formulae (QBF) $\PXQSATQBF$ in prenex normal form of the form
\begin{equation}
\exists \PXQSATBoolVariable{}_{1} \forall \PXQSATBoolVariable{}_{2} \exists \PXQSATBoolVariable{}_{3} \ldots Q \PXQSATBoolVariable{}_{k}
\ [\PXQSATBooleanFormula(\PXQSATBoolVariable{}_{1}, \ldots, \PXQSATBoolVariable{}_{k}) = 1],
\label{eq:QBFKOdd}
\end{equation}
where the sets of Boolean variables $\PXQSATBoolVariable{}_l$, for $l = 1, \ldots, k$, are pairwise disjoint, $\PXQSATBooleanFormula$ is a quantifier-free Boolean formula and $Q$ is $\exists$ if $k$ is odd and $\forall$ otherwise.
For any $\PXQSATQBF \in \PXQSATQBFSet{k}$ defined as \eqref{eq:QBFKOdd}, we write $\PXQSATQBF(\PXQSATTruthAssignment{}_1)$ to denote the QBF obtained by removing the quantifier $\exists \PXQSATBoolVariable{}_1$ and substituting the binary vector $\PXQSATTruthAssignment{}_1$ for $\PXQSATBoolVariable{}_1$:
$$
\forall \PXQSATBoolVariable{}_{2} \exists \PXQSATBoolVariable{}_{3} \ldots Q \PXQSATBoolVariable{}_{k}
\ [\PXQSATBooleanFormula(\PXQSATTruthAssignment{}_{1}, \PXQSATBoolVariable{}_{2}, \ldots, \PXQSATBoolVariable{}_{k}) = 1].
$$
For any $\PXQSATQBF \in \PXQSATQBFSet{k}$ that \PXQSATIsTrue{}, we define $\PXQSATQBFLexMaxSolution(\PXQSATQBF)$ to be the lexicographically maximum binary vector $\PXQSATTruthAssignment{}_1$ such that $\PXQSATQBF(\PXQSATTruthAssignment{}_1)$ \PXQSATIsTrue{}.

For $k \in \PositiveIntegers$, we define
$\DecisionProblemQThreeSAT{k} = \{ H \in \PXQSATQBFSet{k} : \text{$H$ \PXQSATIsTrue{}}\}$, $\SearchProblemQThreeSAT{k}$ as the search problem whose input is $H \in \PXQSATQBFSet{k}$ and requires $\PXQSATQBFLexMaxSolution(\PXQSATQBF)$, or ``no'' if $\PXQSATQBF$ \PXQSATIsNotTrue{}, and $\DecisionProblemLastBitQSAT{k} = \{ \PXQSATQBF \in \PXQSATQBFSet{k} : \text{$\PXQSATQBF$ \PXQSATIsTrue{} and the last component of $\PXQSATQBFLexMaxSolution(\PXQSATQBF)$ is $1$} \}$.
We have that \DecisionProblemQThreeSAT{k} is $\Sigma^p_k$-complete~\MyCite{wrathall1976complete} and $\SearchProblemQThreeSAT{k}$ is $\ClassFDeltaP{k}$-complete~\MyCite{krentel1992generalizations}.
Furthermore, \MyCitet{Krentel}{krentel1992generalizations} shows that the decision of any bit of $\PXQSATQBFLexMaxSolution(\PXQSATQBF)$, the language $\{ (\PXQSATQBF, l) \in \PXQSATQBFSet{k} \times \PositiveIntegers{} : \text{$\PXQSATQBF$ \PXQSATIsTrue{} and the $l$-th component of $\PXQSATQBFLexMaxSolution(\PXQSATQBF)$ is $1$} \}$, is \ClassDeltaP{k}-complete.
It follows that the decision of the least significant bit, the language \DecisionProblemLastBitQSAT{k}, is \ClassDeltaP{k}-complete.

\subsection{Search Problems}\label{sec:HardnessProofs:SearchProblems}

In this section, we study the hardness of $\SearchProblemObj{k}$ for $k \ge 2$.


\subsubsection{Case $k = 2$}
\label{sec:HardnessProofs:SearchProblems:KTwo}

We begin with the following result for bilevel LP.

\begin{theorem}
\label{theorem:search-obj-hardness-bilevel}
The search problem \SearchProblemObj{2} is $\ClassFDeltaP{2}$-hard.
Moreover, this remains valid even when the input is restricted to satisfy conditions~\ref{Condition1}--\ref{Condition3}.
\end{theorem}

We will show this by reducing \SearchProblemThreeSAT{} to \DecisionProblemModifierNoLinkBoxedPoly{\SearchProblemObj{2}}.
To this end, we prepare some auxiliary results.

Given $\PXSATBooleanFormula \in \PXSATBFSet$, \MyCitet{Jeroslow}{Jeroslow1985} constructed a rational bilevel LP instance
\begin{align}
\tag{\EqtagSATBLP{2}{1}{(\PXSATBooleanFormula)}}
\label{eq:SATBLP}
\min_{\PXSATTruthAssignment, \PXSATObjVariable, \PXSATAuxVarOne, \PXSATAuxVarTwo}
\{
\PXSATObjective{}
:
\PXSATAuxVarTwo
\in \mathcal{S}\eqref{eq:SATBLPSecond})
\}
\end{align}
and
\begin{align}
\tag{\EqtagSATBLP{2}{2}{(\PXSATTruthAssignment, \PXSATObjVariable, \PXSATAuxVarOne, \PXSATBooleanFormula)}}
\label{eq:SATBLPSecond}
\min_{\PXSATAuxVarTwo}
\{
(\PXSATCostCoefTwo^{\PXSATBooleanFormula})^{\top} \PXSATAuxVarTwo
:
\PXSATQBFConstraintCoef^{\PXSATBooleanFormula}_{1} \PXSATTruthAssignment
+ \PXSATQBFConstraintCoef^{\PXSATBooleanFormula}_{2} \PXSATObjVariable
+ \PXSATQBFConstraintCoef^{\PXSATBooleanFormula}_{3} \PXSATAuxVarOne
+ \PXSATQBFConstraintCoef^{\PXSATBooleanFormula}_{4} \PXSATAuxVarTwo \ge \PXSATQBFConstraintRHS^{\PXSATBooleanFormula}
\},
\end{align}
where $\OptValFunc\EqrefSATBLP{2}{1}{(\PXSATBooleanFormula)} = \PXSATYesObjVal{}$ if $\PXSATBooleanFormula$ is satisfiable and $\OptValFunc\EqrefSATBLP{2}{1}{(\PXSATBooleanFormula)} = \PXSATNoObjVal{}$ otherwise.
The vector of variables $\PXSATTruthAssignment$ represents a truth assignment to the Boolean variables in $\PXSATBooleanFormula$.
The variable $\PXSATObjVariable$ represents the truth value of $\PXSATBooleanFormula$, whereas $\PXSATAuxVarOne$ is an auxiliary variable that encodes the truth values of the subformulae of $\PXSATBooleanFormula$.
For a fixed binary vector $\PXSATTruthAssignment$, the feasible values of $\PXSATAuxVarOne$ and $\PXSATObjVariable$ are uniquely determined.
The variables $\PXSATAuxVarTwo$ are auxiliary variables introduced to discourage the first player from choosing fractional values.
For any $(\PXSATTruthAssignment, \PXSATObjVariable, \PXSATAuxVarOne)$, there is a unique value of $\PXSATAuxVarTwo$ that is optimal to the second player.
The instance $\EqrefSATBLP{2}{1}{(\PXSATBooleanFormula)}$ satisfies conditions~\ref{Condition1} and \ref{Condition2}.
In particular, the second player's constraints include variable bounds, for example, $\ZeroVector \le \PXSATTruthAssignment \le \OneVector$.
The constructed instance has integer cost coefficients of polynomial magnitude. 
Consequently, by introducing dummy variables, one can assume that \eqref{eq:SATBLP} also satisfies condition~\ref{Condition3}.

Below, we summarize the properties of
$\EqrefSATBLP{2}{1}{(\PXSATBooleanFormula)}$ that will be used later.

\begin{lemma}
\label{lemma:SolutionOfSATBLP}
\mbox{}
\begin{enumerate}
\item
For each binary vector $\PXSATTruthAssignment$, there exists unique vectors $\PXSATObjVariable$, $\PXSATAuxVarOne$ and $\PXSATAuxVarTwo$ such that $(\PXSATTruthAssignment, \PXSATObjVariable, \PXSATAuxVarOne, \PXSATAuxVarTwo) \in \FeasSolSet\EqrefSATBLP{2}{1}{(\PXSATBooleanFormula)}$.
Furthermore, the vector $\PXSATAuxVarOne$ is binary, $\PXSATObjVariable = 1$ if $\PXSATTruthAssignment$ satisfies $\PXSATBooleanFormula$ and $\PXSATObjVariable = 0$ otherwise, and $\PXSATAuxVarTwo = \ZeroVector$.
\item
If $\PXSATBooleanFormula$ is satisfiable, then $\OptValFunc\EqrefSATBLP{2}{1}{(\PXSATBooleanFormula)} = \PXSATYesObjVal{}$, and $(\PXSATTruthAssignment, \PXSATObjVariable, \PXSATAuxVarOne, \PXSATAuxVarTwo) \in \OptSolSet\EqrefSATBLP{2}{1}{(\PXSATBooleanFormula)}$ if and only if $\PXSATTruthAssignment$ is a binary vector satisfying $\PXSATBooleanFormula$.

\item
If $\PXSATBooleanFormula$ is not satisfiable, then $\OptValFunc\EqrefSATBLP{2}{1}{(\PXSATBooleanFormula)} = \PXSATNoObjVal{}$, and $(\PXSATTruthAssignment, \PXSATObjVariable, \PXSATAuxVarOne, \PXSATAuxVarTwo) \in \OptSolSet\EqrefSATBLP{2}{1}{(\PXSATBooleanFormula)}$ if and only if $\PXSATTruthAssignment$ is a binary vector.

\item
Let $(\PXSATTruthAssignment, \PXSATObjVariable, \PXSATAuxVarOne, \PXSATAuxVarTwo) \in \FeasSolSet\EqrefSATBLP{2}{1}{(\PXSATBooleanFormula)}$.
If $\PXSATObjective{} < 2.5$, then there exists a unique binary vector $(\bar{\PXSATTruthAssignment}, \bar{\PXSATObjVariable}, \bar{\PXSATAuxVarOne}, \bar{\PXSATAuxVarTwo})$ that is closest to $(\PXSATTruthAssignment, \PXSATObjVariable, \PXSATAuxVarOne, \PXSATAuxVarTwo)$, and satisfies $(\bar{\PXSATTruthAssignment}, \bar{\PXSATObjVariable}, \bar{\PXSATAuxVarOne}, \bar{\PXSATAuxVarTwo}) \in \FeasSolSet\EqrefSATBLP{2}{1}{(\PXSATBooleanFormula)}$, as well as $1 - \bar{\PXSATObjVariable} + (\PXSATCostCoefOne^{\PXSATBooleanFormula})^{\top} \bar{\PXSATAuxVarTwo} \le \PXSATObjective{}$.
\end{enumerate}
\end{lemma}

\begin{proof}
See Lemma~4.4 of \MyCitet{Jeroslow}{Jeroslow1985}.
\end{proof}

Now, consider the following instance
\begin{align}
\min_{\PXSATTruthAssignment, \PXSATObjVariable, \PXSATAuxVarOne, \PXSATAuxVarTwo, \PXSATOneIndicator}
\ &
\PXSATObjective{}
- \sum_{i = 1}^{\PXSATNVariables} \frac{1}{2^i} \PXSATOneIndicator{}_i
\tag{\EqtagLEXSATBLP{2}{1}{(\PXSATBooleanFormula)}}
\label{eq:LEXSATBLP}
\\
\text{s.t.}
\ &
(\PXSATAuxVarTwo, \PXSATOneIndicator)
\in \mathcal{S}\eqref{eq:LEXSATBLPSecond})
\notag
\end{align}
and
\begin{align}
\min_{\PXSATAuxVarTwo, \PXSATOneIndicator}
\ &
(\PXSATCostCoefTwo^{\PXSATBooleanFormula})^{\top} \PXSATAuxVarTwo
+ \OneVector^{\top} \PXSATOneIndicator
\tag{\EqtagLEXSATBLP{2}{2}{(\PXSATTruthAssignment, \PXSATObjVariable, \PXSATAuxVarOne, \PXSATBooleanFormula)}}
\label{eq:LEXSATBLPSecond}
\\
\text{s.t.}
\ &
\displaystyle
\PXSATQBFConstraintCoef^{\PXSATBooleanFormula}_{1} \PXSATTruthAssignment
+ \PXSATQBFConstraintCoef^{\PXSATBooleanFormula}_{2} \PXSATObjVariable
+ \PXSATQBFConstraintCoef^{\PXSATBooleanFormula}_{3} \PXSATAuxVarOne
+ \PXSATQBFConstraintCoef^{\PXSATBooleanFormula}_{4} \PXSATAuxVarTwo \ge \PXSATQBFConstraintRHS^{\PXSATBooleanFormula},
\notag
\\
&
2 \PXSATTruthAssignment - \OneVector \le \PXSATOneIndicator,
\notag
\\
&
\ZeroVector \le \PXSATOneIndicator \le \OneVector,
\notag
\end{align}
where $\PXSATNVariables$ is the number of Boolean variables in $\PXSATBooleanFormula$.
We have $\OptValFunc\EqrefLEXSATBLP{2}{1}{(\PXSATBooleanFormula)} \le \OptValFunc\EqrefSATBLP{2}{1}{(\PXSATBooleanFormula)}$.
Below, we show that the optimal solution of $\EqrefLEXSATBLP{2}{1}{(\PXSATBooleanFormula)}$ is the lexicographically maximum binary vector satisfying $\PXSATBooleanFormula$.

\begin{lemma}
\label{lemma:SolutionIsLexicographicallyMaximum}
\begin{enumerate}
\item
Suppose $\PXSATBooleanFormula$ is satisfiable.
Then,
$$
\OptSolSet\EqrefLEXSATBLP{2}{1}{(\PXSATBooleanFormula)} = \{ (\PXSATTruthAssignment, \PXSATObjVariable, \PXSATAuxVarOne, \PXSATAuxVarTwo, \PXSATOneIndicator) : \text{ $\PXSATTruthAssignment = \PXSATOneIndicator = \PXSATLexMaxSolution(\PXSATBooleanFormula)$, $(\PXSATTruthAssignment, \PXSATObjVariable, \PXSATAuxVarOne, \PXSATAuxVarTwo) \in \FeasSolSet\EqrefSATBLP{2}{1}{(\PXSATBooleanFormula)}$} \}.
$$
and
\begin{equation*}
\OptValFunc\EqrefLEXSATBLP{2}{1}{(\PXSATBooleanFormula)} = - \sum_{i=1}^{\PXSATNVariables} \frac{1}{2^{i}} \PXSATLexMaxSolution(\PXSATBooleanFormula)_i.
\end{equation*}

\item
Suppose $\PXSATBooleanFormula$ is not satisfiable.
Then,
$$
\OptSolSet\EqrefLEXSATBLP{2}{1}{(\PXSATBooleanFormula)} = \{ (\PXSATTruthAssignment, \PXSATObjVariable, \PXSATAuxVarOne, \PXSATAuxVarTwo, \PXSATOneIndicator) : \text{ $\PXSATTruthAssignment = \PXSATOneIndicator = \OneVector$, $(\PXSATTruthAssignment, \PXSATObjVariable, \PXSATAuxVarOne, \PXSATAuxVarTwo) \in \FeasSolSet\EqrefSATBLP{2}{1}{(\PXSATBooleanFormula)}$} \}.
$$
and
\begin{equation*}
\OptValFunc\EqrefLEXSATBLP{2}{1}{(\PXSATBooleanFormula)} = \frac{1}{2^{\PXSATNVariables}}.
\end{equation*}
\end{enumerate}
\end{lemma}

\begin{proof}
We only prove assertion~(i) since the proof of assertion~(ii) is similar.

Suppose $\PXSATBooleanFormula$ is satisfiable.
By Lemma~\ref{lemma:SolutionOfSATBLP}, we have $\OptValFunc\EqrefLEXSATBLP{2}{1}{(\PXSATBooleanFormula)} \le \OptValFunc\EqrefSATBLP{2}{1}{(\PXSATBooleanFormula)} = 0$.
It follows that $(\PXSATTruthAssignment, \PXSATObjVariable, \PXSATAuxVarOne, \PXSATAuxVarTwo, \PXSATOneIndicator) \in \OptSolSet\EqrefLEXSATBLP{2}{1}{(\PXSATBooleanFormula)}$ implies
\begin{equation}
\PXSATObjective{}
< 1,
\label{eq:Lemma:SolutionOfCompactLextSATBLP:StepOne}
\end{equation}
for otherwise the objective value of  $(\PXSATTruthAssignment, \PXSATObjVariable, \PXSATAuxVarOne, \PXSATAuxVarTwo, \PXSATOneIndicator)$ is
$$
\PXSATObjective{}
- \sum_{i = 1}^{\PXSATNVariables} \frac{1}{2^{i}} \PXSATOneIndicator{}_{i}
\ge
1 - \sum_{i = 1}^{\PXSATNVariables} \frac{1}{2^{i}} \PXSATOneIndicator{}_{i}
>
0
\ge \OptValFunc\EqrefLEXSATBLP{2}{1}{(\PXSATBooleanFormula)}.
$$

Now, we show that
\begin{equation}
(\PXSATTruthAssignment, \PXSATObjVariable, \PXSATAuxVarOne, \PXSATAuxVarTwo, \PXSATOneIndicator) \in \OptSolSet\EqrefLEXSATBLP{2}{1}{(\PXSATBooleanFormula)}
\Rightarrow
(\PXSATTruthAssignment, \PXSATObjVariable, \PXSATAuxVarOne, \PXSATAuxVarTwo) \in \OptSolSet\EqrefSATBLP{2}{1}{(\PXSATBooleanFormula)}.
\label{eq:Lemma:SolutionOfCompactLextSATBLP:StepTwo}
\end{equation}
For the sake of contradiction, suppose otherwise and let $(\PXSATTruthAssignment, \PXSATObjVariable, \PXSATAuxVarOne, \PXSATAuxVarTwo, \PXSATOneIndicator) \in \OptSolSet\EqrefLEXSATBLP{2}{1}{(\PXSATBooleanFormula)}$ be such that $(\PXSATTruthAssignment, \PXSATObjVariable, \PXSATAuxVarOne, \PXSATAuxVarTwo) \not\in \OptSolSet\EqrefSATBLP{2}{1}{(\PXSATBooleanFormula)}$.
Then, in light of Lemma~\ref{lemma:SolutionOfSATBLP} and \eqref{eq:Lemma:SolutionOfCompactLextSATBLP:StepOne}, there exists a unique binary vector $(\bar{\PXSATTruthAssignment}, \bar{\PXSATObjVariable}, \bar{\PXSATAuxVarOne}, \bar{\PXSATAuxVarTwo})$ that is the closest to $(\PXSATTruthAssignment, \PXSATObjVariable, \PXSATAuxVarOne, \PXSATAuxVarTwo)$.
Moreover, one can show that $(\bar{\PXSATTruthAssignment}, \bar{\PXSATObjVariable}, \bar{\PXSATAuxVarOne}, \bar{\PXSATAuxVarTwo}, \bar{\PXSATTruthAssignment})$ is feasible to $\EqrefLEXSATBLP{2}{1}{(\PXSATBooleanFormula)}$ and has a strictly better objective value, contradicting the assumption.
Thus, \eqref{eq:Lemma:SolutionOfCompactLextSATBLP:StepTwo} holds.

For any $(\PXSATTruthAssignment, \PXSATObjVariable, \PXSATAuxVarOne, \PXSATAuxVarTwo, \PXSATOneIndicator) \in \OptSolSet\EqrefLEXSATBLP{2}{1}{(\PXSATBooleanFormula)}$, \eqref{eq:Lemma:SolutionOfCompactLextSATBLP:StepTwo} implies $\PXSATTruthAssignment$ is a binary vector satisfying $\PXSATBooleanFormula$ and $\PXSATObjective{} = 0$.
Thus,
\begin{align*}
\OptValFunc\EqrefLEXSATBLP{2}{1}{(\PXSATBooleanFormula)}
=
\min_{\PXSATTruthAssignment} \left\{- \sum_{i = 1}^{\PXSATNVariables} \frac{1}{2^{i}} \PXSATTruthAssignment{}_i : \text{$\PXSATTruthAssignment$ satisfies $\PXSATBooleanFormula$} \right\},
\end{align*}
which is minimized by the lexicographically maximum binary vector satisfying $\PXSATBooleanFormula$.
\end{proof}

Instance~\EqrefLEXSATBLP{2}{1}{(\PXSATBooleanFormula)} uses fractional values of exponentially small magnitude.
We can remove them by introducing auxiliary variables.
Consider the following bilevel LP:
\begin{align}
\inf_{\PXSATTruthAssignment, \PXSATObjVariable, \PXSATAuxVarOne, \PXSATAuxVarTwo, \PXSATOneIndicator, \PXSATOneIndicatorExpanded}
\ &
\PXSATObjective
- \sum_{i = 1}^{\PXSATNVariables} \PXSATOneIndicatorExpanded{}_{i i}
\tag{\EqtagCompactLEXSATBLP{2}{1}{(\PXSATBooleanFormula)}}
\label{eq:CompactLEXSATBLP}
\\
\text{s.t.}
\ &
(\PXSATAuxVarTwo, \PXSATOneIndicator, \PXSATOneIndicatorExpanded)
\in \mathcal{S}\eqref{eq:LEXSATBLPSecond})
\notag
\end{align}
and
\begin{align}
\inf_{\PXSATAuxVarTwo, \PXSATOneIndicator, \PXSATOneIndicatorExpanded}
\ &
(\PXSATCostCoefTwo^{\PXSATBooleanFormula})^{\top} \PXSATAuxVarTwo
+ \OneVector^{\top} \PXSATOneIndicator
\tag{\EqtagCompactLEXSATBLP{2}{2}{(\PXSATTruthAssignment, \PXSATObjVariable, \PXSATAuxVarOne)}}
\label{eq:CompactLEXSATBLPSecond}
\\
\text{s.t.}
\ &
\mathmakebox[0pt][l]{
\PXSATQBFConstraintCoef^{\PXSATBooleanFormula}_{1} \PXSATTruthAssignment
+ \PXSATQBFConstraintCoef^{\PXSATBooleanFormula}_{2} \PXSATObjVariable
+ \PXSATQBFConstraintCoef^{\PXSATBooleanFormula}_{3} \PXSATAuxVarOne
+ \PXSATQBFConstraintCoef^{\PXSATBooleanFormula}_{4} \PXSATAuxVarTwo \ge \PXSATQBFConstraintRHS^{\PXSATBooleanFormula},
}
\notag
\\
&
2 \PXSATTruthAssignment - \OneVector \le \PXSATOneIndicator,
\notag
\\
&
2 \PXSATOneIndicatorExpanded{}_{i 1} = \PXSATOneIndicator{}_i,
&& \forall i = 1, \ldots, \PXSATNVariables,
\notag
\\
&
2 \PXSATOneIndicatorExpanded{}_{i p + 1} = \PXSATOneIndicatorExpanded{}_{i p},
&& \forall i = 1, \ldots, \PXSATNVariables, p = 1, \ldots, \PXSATNVariables - 1,
\notag
\\
&
\ZeroVector \le \PXSATOneIndicator \le \OneVector, \ZeroVector \le \PXSATOneIndicatorExpanded \le \OneVector.
\notag
\end{align}
It is straightforward to observe that $(\PXSATTruthAssignment, \PXSATObjVariable, \PXSATAuxVarOne, \PXSATAuxVarTwo, \PXSATOneIndicator) \in \mathcal{F}\EqrefLEXSATBLP{2}{1}{(\PXSATBooleanFormula)}$ if and only if $(\PXSATTruthAssignment, \PXSATObjVariable, \PXSATAuxVarOne, \PXSATAuxVarTwo, \PXSATOneIndicator, \PXSATOneIndicatorExpanded) \in \mathcal{F}\EqrefCompactLEXSATBLP{2}{1}{(\PXSATBooleanFormula)}$, and for any $(\PXSATTruthAssignment, \PXSATObjVariable, \PXSATAuxVarOne, \PXSATAuxVarTwo, \PXSATOneIndicator, \PXSATOneIndicatorExpanded) \in \mathcal{F}\EqrefCompactLEXSATBLP{2}{1}{(\PXSATBooleanFormula)}$, we have $\PXSATOneIndicatorExpanded{}_{i i} = \PXSATOneIndicator{}_i / 2^{i}$ for all $i$.
Moreover, \EqrefCompactLEXSATBLP{2}{1}{(\PXSATBooleanFormula)} satisfies conditions~\ref{Condition1}--\ref{Condition3}.
We are now ready to prove the hardness of \SearchProblemObj{2}.

\begin{proof}[Proof of Theorem~\ref{theorem:search-obj-hardness-bilevel}]
We need only show that \DecisionProblemModifierNoLinkBoxedPoly{\SearchProblemObj{2}} is \ClassDeltaP{2}-hard.
Let $\PXSATBooleanFormula \in \PXSATBFSet$, and construct \EqrefCompactLEXSATBLP{2}{1}{(\PXSATBooleanFormula)}.
Note that \EqrefCompactLEXSATBLP{2}{1}{(\PXSATBooleanFormula)} is a rational \DecisionProblemModifierNoLinkBoxedPoly{\KLPInstances{2}} instance.
By Lemma~\ref{lemma:SolutionIsLexicographicallyMaximum}, if $\OptValFunc\EqrefCompactLEXSATBLP{2}{1}{(\PXSATBooleanFormula)} > 0$, $\PXSATBooleanFormula$ is not satisfiable.
Otherwise, again using Lemma~\ref{lemma:SolutionIsLexicographicallyMaximum}, one can compute the lexicographically maximum binary vector satisfying $\PXSATBooleanFormula$.
From this argument, it is trivial to construct a metric reduction from \SearchProblemThreeSAT{} to \DecisionProblemModifierNoLinkBoxedPoly{\SearchProblemObj{2}}~\cite{krentel1986complexity}.
\end{proof}

\subsubsection{Case $k \ge 3$}
\label{sec:hardnessproofs:searchproblems:kgethree}

Next, we show the following.

\begin{theorem}
\label{theorem:search-obj-hardness-general}
Let $k \ge 3$.
The search problem \SearchProblemObj{k} is $\ClassFDeltaP{k}$-hard, and this remains valid even when the input is restricted to satisfy conditions~\ref{Condition1}--\ref{Condition3}.
\end{theorem}

\MyCitet{Jeroslow}{Jeroslow1985} extends the instance~$\EqrefSATBLP{2}{1}{(\PXSATBooleanFormula)}$ to handle alternative quantifiers.
Given $k \ge 3$ and $\PXQSATQBF \in \PXQSATQBFSet{k-1}$, construct a \DecisionProblemModifierNoLinkBoxedPoly{\KLPInstances{k}} instance as follows.
For $l = 1, \ldots, k - 1$, let $\PXQSATNVariables{}_l$ be the number of variables in the $l$-th set of Boolean variables ($\PXQSATBoolVariable{}_l$ in \eqref{eq:QBFKOdd}).
The first player's decision variable is $\PXQSATTruthAssignment{}_1 \in \mathbb{R}^{n_1}$, the second player's decision variable is $\PXQSATTruthAssignment{}_2 \in \mathbb{R}^{n_2}$, and similarly, the $l$-th player for $l=3,\ldots,k-2$ chooses $\PXQSATTruthAssignment{}_l \in \mathbb{R}^{n_l}$.
The $(k-1)$-th player chooses $\PXQSATTruthAssignment{}_{k-1} \in \mathbb{R}^{n_{k-1}}$, $\PXQSATObjVariable \in \mathbb{R}$, and $\PXQSATAuxVarOne$, while the $k$-th player chooses $\PXQSATAuxVarPen \in \mathbb{R}^k$ and $\PXQSATAuxVarTwo$.
The instance is defined by
\begin{align}
&
\inf_{\PXQSATTruthAssignment{}_{l}, \ldots, \PXQSATTruthAssignment{}_{k-1}, \PXQSATObjVariable, \PXQSATAuxVarOne, \PXQSATAuxVarPen, \PXQSATAuxVarTwo}
\left\{
\theta^1_k(\PXQSATObjVariable)
+
2 \PXQSATAuxVarPen{}_1
:
(\PXQSATTruthAssignment{}_2, \ldots, \PXQSATTruthAssignment{}_{k-1}, \PXQSATObjVariable, \PXQSATAuxVarOne, \PXQSATAuxVarPen, \PXQSATAuxVarTwo)
\in \OptSolSet(\EqtagNewQLPOriginal{k}{2}{(\PXQSATTruthAssignment{}_1, \PXQSATQBF)})
\right\}
\tag{\EqtagNewQLPOriginal{k}{1}{(\PXQSATQBF)}},
\label{eq:QLPOriginal}
\end{align}
$$
\vdots
$$
\begin{align}
&
\inf_{\PXQSATTruthAssignment{}_{k-2}, \PXQSATTruthAssignment{}_{k-1}, \PXQSATObjVariable, \PXQSATAuxVarOne, \PXQSATAuxVarPen, \PXQSATAuxVarTwo}
\left\{
\theta^{k-2}_k(\PXQSATObjVariable)
+
2 \PXQSATAuxVarPen{}_{k-2}
:
(\PXQSATTruthAssignment{}_{k-1}, \PXQSATObjVariable, \PXQSATAuxVarOne, \PXQSATAuxVarPen, \PXQSATAuxVarTwo)
\in \mathcal{S}\eqref{eq:QLPOriginalSecondLast}
\right\}
\tag{\EqtagNewQLPOriginal{k}{k-2}{(\PXQSATTruthAssignment{}_{1}, \ldots, \PXQSATTruthAssignment{}_{k-3}, \PXQSATQBF)}},
\label{eq:QLPOriginalThirdLast}
\end{align}
\begin{align}
&
\inf_{\PXQSATTruthAssignment{}_{k-1}, \PXQSATObjVariable, \PXQSATAuxVarOne, \PXQSATAuxVarPen, \PXQSATAuxVarTwo}
\left\{
\theta^{k-1}_k(\PXQSATObjVariable)
+ (\PXQSATCostCoefOne^{\PXQSATQBF})^{\top} \PXQSATAuxVarTwo
:
\begin{array}{ll}
(\PXQSATAuxVarPen, \PXQSATAuxVarTwo)
\in \mathcal{S}\eqref{eq:QLPOriginalLast})
\end{array}
\right\},
\tag{\EqtagNewQLPOriginal{k}{k - 1}{(\PXQSATTruthAssignment{}_1, \ldots, \PXQSATTruthAssignment{}_{k-2}, \PXQSATQBF)}}
\label{eq:QLPOriginalSecondLast}
\end{align}
and
\begin{align}
&
\inf_{\PXQSATAuxVarPen, \PXQSATAuxVarTwo}
\left\{
(\PXQSATCostCoefTwo^{\PXQSATQBF})^{\top} \PXQSATAuxVarTwo
:
\begin{array}{ll}
\displaystyle
\sum_{l=1}^{k-1} \PXQSATQBFConstraintCoef^{\PXQSATQBF}_{l} \PXQSATTruthAssignment{}_{l}
+ \PXQSATQBFConstraintCoef^{\PXQSATQBF}_{k} \PXQSATObjVariable
+ \PXQSATQBFConstraintCoef^{\PXQSATQBF}_{k+1} \PXQSATAuxVarOne
+ \PXQSATQBFConstraintCoef^{\PXQSATQBF}_{k+2} \PXQSATAuxVarPen
+ \PXQSATQBFConstraintCoef^{\PXQSATQBF}_{k+3} \PXQSATAuxVarTwo \ge \PXQSATQBFConstraintRHS^{\PXQSATQBF}
\end{array}
\right\},
\tag{\EqtagNewQLPOriginal{k}{k}{(\PXQSATTruthAssignment{}_{1}, \ldots, \PXQSATTruthAssignment{}_{k-1}, \PXQSATObjVariable, \PXQSATAuxVarOne, \PXQSATQBF)}}
\label{eq:QLPOriginalLast}
\end{align}
where $\theta^l_k(\PXQSATObjVariable) = 1 - \PXQSATObjVariable$ if $k - l$ is odd and $\theta^l_k(\PXQSATObjVariable) = \PXQSATObjVariable$ otherwise.
Instance $\EqrefNewQLPOriginal{k}{1}{(\PXQSATQBF)}$ satisfies conditions~\ref{Condition1}--\ref{Condition3} (after appending polynomially many dummy variables).
Furthermore, we have that $\OptValFunc\EqrefNewQLPOriginal{k}{1}{(\PXQSATQBF)} = \PXQSATYesObjVal{}$ if $\PXQSATQBF$ \PXQSATIsTrue{} and $\OptValFunc\EqrefNewQLPOriginal{k}{1}{(\PXQSATQBF)} = \PXQSATNoObjVal{}$ otherwise.
This instance is constructed so that the last two players evaluate the first $(k - 2)$ players' choices in the following sense.
For example, suppose $k = 4$ and $\PXQSATQBF \in \PXQSATQBFSet{3}$ is given as
$\exists \PXQSATBoolVariable{}_{1}$ $\forall \PXQSATBoolVariable{}_{2}$ $\exists \PXQSATBoolVariable{}_{3}$ $\PXQSATBooleanFormula(\PXQSATBoolVariable{}_{1}, \PXQSATBoolVariable{}_{2}, \PXQSATBoolVariable{}_{3})$.
If the first two players choose binary vectors $\PXQSATTruthAssignment{}_1$ and $\PXQSATTruthAssignment{}_2$, at the optimality of the third player, we have $\theta_4^1(\PXQSATObjVariable) = \PXQSATYesObjVal{}$ (and $\theta_4^2(\PXQSATObjVariable) = \PXQSATNoObjVal{}$) if $\exists \PXQSATBoolVariable{}_{3} \ \PXQSATBooleanFormula(\PXQSATTruthAssignment{}_{1}, \PXQSATTruthAssignment{}_{2}, \PXQSATBoolVariable{}_{3})$ \PXQSATIsTrue{}, and we have $\theta_4^1(\PXQSATObjVariable) = \PXQSATNoObjVal{}$ (and $\theta_4^2(\PXQSATObjVariable) = \PXQSATYesObjVal{}$) otherwise.
In other words, the first player tries to find $\PXQSATTruthAssignment{}_1$ such that for any choice $\PXQSATTruthAssignment{}_2$ of the second player, $
\exists \PXQSATBoolVariable{}_{3}
\ \PXQSATBooleanFormula(\PXQSATTruthAssignment{}_{1}, \PXQSATTruthAssignment{}_{2}, \PXQSATBoolVariable{}_{3})
$
\PXQSATIsTrue{}, and achieves the objective value of $\PXQSATYesObjVal{}$ in this case.

Variable $\PXQSATAuxVarPen$ is used as a penalty for the use of fractional values.
If the first player chooses a fractional value, the optimality of the last two players implies that $\PXQSATAuxVarPen{}_1 = 1$.
If the first player chooses a binary vector, the other players are indifferent to the values of $\PXQSATAuxVarPen{}_1$, and due to the optimistic assumption, the first player can assume $\PXQSATAuxVarPen{}_1 = 0$ is chosen.
We note that, in the original construction of \MyCitet{Jeroslow}{Jeroslow1985}, the penalty term is defined separately for each component of $\PXQSATTruthAssignment{}_1$.
However, it is straightforward to aggregate these componentwise penalties into a single term, as done here.

Formally, we have the following.
\begin{lemma}
\label{lemma:ValueOfQLP}
Suppose $k \ge 2$ and $\PXQSATQBF \in \PXQSATQBFSet{k-1}$.
\begin{enumerate}
\item
If $\PXQSATTruthAssignment{}_1$ is binary and
$\PXQSATQBF(\PXQSATTruthAssignment{}_1)$ \PXQSATIsTrue{}, then $\OptSolSet\EqrefNewQLPOriginal{k}{2}{(\PXQSATTruthAssignment{}_1, \PXQSATQBF)} \neq \emptyset$,
and any
\[
(\PXQSATTruthAssignment{}_2, \ldots, \PXQSATTruthAssignment{}_{k-1},
\PXQSATObjVariable, \PXQSATAuxVarOne,
\PXQSATAuxVarPen, \PXQSATAuxVarTwo)
\in
\OptSolSet\EqrefNewQLPOriginal{k}{2}{(\PXQSATTruthAssignment{}_1, \PXQSATQBF)}
\]
satisfies $\theta_k^1(\PXQSATObjVariable) = 0$.

\item
If $\PXQSATTruthAssignment{}_1$ is binary and
$\PXQSATQBF(\PXQSATTruthAssignment{}_1)$ \PXQSATIsNotTrue{}, then $\OptSolSet\EqrefNewQLPOriginal{k}{2}{(\PXQSATTruthAssignment{}_1, \PXQSATQBF)} \neq \emptyset$, and any
\[
(\PXQSATTruthAssignment{}_2, \ldots, \PXQSATTruthAssignment{}_{k-1},
\PXQSATObjVariable, \PXQSATAuxVarOne,
\PXQSATAuxVarPen, \PXQSATAuxVarTwo)
\in
\OptSolSet\EqrefNewQLPOriginal{k}{2}{(\PXQSATTruthAssignment{}_1, \PXQSATQBF)}
\]
satisfies $\theta_k^1(\PXQSATObjVariable) = 1$.

\item
Let $l \le k - 2$ and suppose
$\PXQSATTruthAssignment{}_l \notin \{0,1\}^{n_l}$.
If
\[
(\PXQSATTruthAssignment{}_{k-1}, \PXQSATObjVariable,
\PXQSATAuxVarOne, \PXQSATAuxVarPen, \PXQSATAuxVarTwo)
\in
\OptSolSet\EqrefNewQLPOriginal{k}{k-1}{
(\PXQSATTruthAssignment{}_1, \ldots,
\PXQSATTruthAssignment{}_{k-2}, \PXQSATQBF)},
\]
then $\PXQSATAuxVarPen{}_l = 1$.

\item
Let $l \le k - 2$ and suppose
$\PXQSATTruthAssignment{}_l \in \{0,1\}^{n_l}$.
If
\[
(\PXQSATTruthAssignment{}_{k-1}, \PXQSATObjVariable,
\PXQSATAuxVarOne, \PXQSATAuxVarPen, \PXQSATAuxVarTwo)
\in
\OptSolSet\EqrefNewQLPOriginal{k}{k-1}{
(\PXQSATTruthAssignment{}_1, \ldots,
\PXQSATTruthAssignment{}_{k-2}, \PXQSATQBF)},
\]
then
\[
(\PXQSATTruthAssignment{}_{k-1}, \PXQSATObjVariable,
\PXQSATAuxVarOne, \PXQSATAuxVarPen', \PXQSATAuxVarTwo)
\in
\OptSolSet\EqrefNewQLPOriginal{k}{k-1}{
(\PXQSATTruthAssignment{}_1, \ldots,
\PXQSATTruthAssignment{}_{k-2}, \PXQSATQBF)},
\]
where
\[
\PXQSATAuxVarPen' =
(\PXQSATAuxVarPen{}_1, \ldots, \PXQSATAuxVarPen{}_{l-1},
0,
\PXQSATAuxVarPen{}_{l+1}, \ldots,
\PXQSATAuxVarPen{}_{k-2}).
\]
\end{enumerate}
\end{lemma}

Now, consider the following $k$-level instance, which is obtained by modifying the objective function of the first player in \eqref{eq:QLPOriginal}:
\begin{align}
&
\inf_{\PXQSATTruthAssignment{}_{l}, \ldots, \PXQSATTruthAssignment{}_{k-1}, \PXQSATObjVariable, \PXQSATAuxVarOne, \PXQSATAuxVarPen, \PXQSATAuxVarTwo}
\left\{
\theta^1_k(\PXQSATObjVariable)
+
2 \PXQSATAuxVarPen{}_1
-
\sum_{i = 1}^{\PXQSATNVariables{}_1} \PXQSATTruthAssignment{}_{1 i} / 2^{i}
:
\begin{array}{l}
(\PXQSATTruthAssignment{}_{2}, \ldots, \PXQSATTruthAssignment{}_{k-1}, \PXQSATObjVariable, \PXQSATAuxVarOne, \PXQSATAuxVarPen, \PXQSATAuxVarTwo)
\\
\qquad\qquad
\in \mathcal{S}(\EqtagNewQLPOriginal{k}{2}{(\PXQSATTruthAssignment{}_1, \PXQSATQBF)})
\end{array}
\right\}
\tag{\EqtagNewQLPSearch{k}{1}{(\PXQSATQBF)}},
\label{eq:QLPSearch}
\end{align}
We have a result analogous to Lemma~\ref{lemma:SolutionIsLexicographicallyMaximum}.
\begin{lemma}
\label{lemma:theorem:search-obj-hardness-general}
\mbox{}
\begin{enumerate}
\item
Suppose $\PXQSATQBF$ \PXQSATIsTrue{}.
Then, $\OptSolSet\eqref{eq:QLPSearch}$ is nonempty and $(\PXQSATTruthAssignment{}_{l}$, $\ldots$, $\PXQSATTruthAssignment{}_{k-1}$, $\PXQSATObjVariable$, $\PXQSATAuxVarOne$, $\PXQSATAuxVarPen$, $\PXQSATAuxVarTwo) \in \OptSolSet\eqref{eq:QLPSearch}$ implies $\PXQSATTruthAssignment{}_{1} = \PXQSATQBFLexMaxSolution(\PXQSATQBF)$, $\theta_k^1(\PXQSATObjVariable) = 0$ and $\PXQSATAuxVarPen{}_1 = 0$.
In particular, $\OptValFunc\eqref{eq:QLPSearch} = - \sum_{i = 1}^{\PXQSATNVariables{}_1} \PXQSATQBFLexMaxSolution(\PXQSATQBF)_i / 2^i$.
\item
Suppose $\PXQSATQBF$ \PXQSATIsNotTrue{}.
Then, $\OptSolSet\eqref{eq:QLPSearch}$ is nonempty and $(\PXQSATTruthAssignment{}_{l}$, $\ldots$, $\PXQSATTruthAssignment{}_{k-1}$, $\PXQSATObjVariable$, $\PXQSATAuxVarOne$, $\PXQSATAuxVarPen$, $\PXQSATAuxVarTwo) \in \OptSolSet\eqref{eq:QLPSearch}$ implies $\PXQSATTruthAssignment{}_{1} = \OneVector$, $\theta_k^1(\PXQSATObjVariable) = 1$ and $\PXQSATAuxVarPen{}_1 = 0$.
In particular, $\OptValFunc\eqref{eq:QLPSearch} = 1 / 2^{\PXQSATNVariables{}_1}$.
\end{enumerate}
\end{lemma}

\begin{proof}
The proofs are similar to those of Lemma~\ref{lemma:SolutionIsLexicographicallyMaximum}. Therefore, we only provide an outline of that for assertion~(i).
By Lemma~\ref{lemma:ValueOfQLP}, we have $\OptValFunc\eqref{eq:QLPSearch} \le \OptValFunc\EqrefNewQLPOriginal{k}{1}{(\PXQSATQBF)} = \PXQSATYesObjVal{}$.
It follows that any feasible solution such that $\PXQSATTruthAssignment{}_1$ is fractional or $\PXQSATQBF(\PXQSATTruthAssignment{}_1)$ \PXQSATIsNotTrue{} is suboptimal.
Thus, for any optimal solution, we have that $\PXQSATTruthAssignment{}_1$ is a binary vector such that $\PXQSATQBF(\PXQSATTruthAssignment{}_1)$ \PXQSATIsTrue{}, $\theta_k^1(\PXQSATObjVariable) = \PXQSATYesObjVal{}$ and $\PXQSATAuxVarPen{}_1 = 0$.
Therefore,
$$
\OptValFunc\eqref{eq:QLPSearch} = \min_{\PXQSATTruthAssignment{}_1} \left\{ - \sum_{i = 1}^{n_1} \frac{\PXQSATTruthAssignment{}_{1 i}}{ 2^{i} } : \PXQSATTruthAssignment{}_1 \in \{0, 1\}^{n_1} \text{ such that $\PXQSATQBF(\PXQSATTruthAssignment{}_1)$ \PXQSATIsTrue{}} \right\}.
$$
Now, the claim follows immediately.
\end{proof}

Next, we prove that \SearchProblemObj{k} is \ClassFDeltaP{k}-hard for $k \ge 3$.

\begin{proof}[Proof of Theorem~\ref{theorem:search-obj-hardness-general}]
We need only show that \DecisionProblemModifierNoLinkBoxedPoly{\SearchProblemObj{k}} is \ClassDeltaP{k}-hard.
The proof is similar to that of Theorem~\ref{theorem:search-obj-hardness-bilevel}, but uses Lemma~\ref{lemma:theorem:search-obj-hardness-general} in place of Lemma~\ref{lemma:SolutionIsLexicographicallyMaximum}.
Note that \EqrefNewQLPSearch{k}{1}{(\PXQSATQBF)} has fractional coefficients, but one can obtain an equivalent rational \DecisionProblemModifierNoLinkBoxedPoly{\KLPInstances{k}} instance as in the proof of Theorem~\ref{theorem:search-obj-hardness-bilevel}.
\end{proof}

\subsection{Attainability Problems}\label{sec:HardnessProofs:AttainabilityProblems}

In this section, we study the hardness of \DecisionProblemAttain{k}.
We begin with special cases in \KLPInstances{2} and \KLPInstances{3}.

\begin{theorem}
\label{theorem:hardness:attainability:bilevel}
\mbox{}
\begin{enumerate}
\item
The decision problem \DecisionProblemModifierPoly{\DecisionProblemAttain{2}} is \ClassDP{}-hard.
\item
The decision problem \DecisionProblemModifierBoxedPoly{\DecisionProblemAttain{2}} is \ClassNP{}-hard.
\item
The decision problem \DecisionProblemModifierNoLinkPoly{\DecisionProblemAttain{2}} is \ClassCoNP{}-hard.
\item
The decision problem \DecisionProblemModifierNoLinkPoly{\DecisionProblemAttain{3}} is \ClassPiP{2}-hard.
\end{enumerate}
\end{theorem}

We will use the following lemma, whose proof is given in~\MyCite{Sugishita2026}.

\begin{lemma}
\label{lemma:homogeneous-klp-without-linking-constraints}
Let $(A, b, c)$ be a rational \DecisionProblemModifierNoLink{\KLPInstances{k}} instance~\eqref{LabelGeneralKLP}.
Then, for any $\lambda \in \mathbb{R}$ such that $\lambda > 0$,
$\mathcal{F}\EqrefGeneralKLP{k}{1}{(A, \lambda b, c)} = \lambda \mathcal{F}\eqref{LabelGeneralKLP}$.
\end{lemma}

\begin{proof}[Proof of Theorem~\ref{theorem:hardness:attainability:bilevel}]
Since the proofs are similar, we only prove assertion~(i).

We use the following language called \SATUNSAT{}: A pair of Boolean formulae $(\PXSATBooleanFormula^1, \PXSATBooleanFormula^2) \in \PXSATBFSet{} \times \PXSATBFSet{}$ is in \SATUNSAT{} if and only if $\PXSATBooleanFormula^1$ is satisfiable and $\PXSATBooleanFormula^2$ is unsatisfiable.
It is shown by \MyCitet{Papadimitriou and Yannakakis}{papadimitriou1982complexity} that \SATUNSAT{} is \ClassDP{}-complete.

For each $i = 1, 2$, let $\PXSATBooleanFormula^i$ be a Boolean formula in $\PXSATBFSet{}$, and let $(A^i, b^i, c^i)$ be a \DecisionProblemModifierNoLinkBoxedPoly{\KLPInstances{2}} instance whose optimal objective value is $-1$ if $\PXSATBooleanFormula^i$ is satisfiable and $0$ otherwise (one can obtain this instance modifying \EqrefSATBLP{2}{1}{(\PXSATBooleanFormula^i)} appropriately).
Consider
\begin{align*}
\min_{\substack{x_1^1, x_2^1,\\x_1^2, x_2^2, s^2}}
\left\{
c_{1 1}^2 x_1^2 + c_{1 2}^2 x_2^2
:
\begin{array}{l}
c_{1 1}^1 x_1^1 + c_{1 2}^1 x_2^1 \le -1,
\\
x_2^1 \in \OptSolSet\EqrefGeneralKLP{2}{2}{(x_1^1, A^1, b^1, c^1)},
\\
x_2^2 \in \OptSolSet\EqrefGeneralKLP{2}{2}{(x_1^2, A^2, s^2 b^2, c^2)}
\end{array}
\right\}.
\end{align*}
By Lemmata~\ref{lemma:SolutionOfSATBLP} and~\ref{lemma:homogeneous-klp-without-linking-constraints}, this has an optimal solution if and only if $\PXSATBooleanFormula^1$ is satisfiable and $\PXSATBooleanFormula^2$ is not satisfiable.
We can rewrite it as an equivalent rational \DecisionProblemModifierPoly{\KLPInstances{2}} instance.
Thus, \SATUNSAT{} $\le_l$ \DecisionProblemModifierPoly{\DecisionProblemAttain{2}}.
\end{proof}

The next results concern \KLPInstances{k} with $k \ge 3$.

\begin{theorem}
\label{theorem:hardness:attainability}
The following statements hold.
\begin{enumerate}
\item
For $k \ge 3$, the decision problem \DecisionProblemModifierBoxedPoly{\DecisionProblemAttain{k}} is $\ClassDeltaP{k}$-hard.
\item
For $k \ge 4$, the decision problem \DecisionProblemModifierNoLinkBoxedPoly{\DecisionProblemAttain{k}} is $\ClassDeltaP{k}$-hard.
\end{enumerate}
\end{theorem}


\begin{proof}
(i)
We prove for the case $k = 3$.
It is straightforward to extend the proof for $k \ge 4$.

Let $\PXQSATQBF \in \PXQSATQBFSet{2}$ and define the following rational \DecisionProblemModifierBoxed{\KLPInstances{3}} instance, which is obtained by combining \eqref{eq:QLPSearch} and \EqrefExampleTwo{3}{1}{}:
\begin{align}
\inf_{\substack{
\PXQSATTruthAssignment{}_1, \PXQSATTruthAssignment{}_2, \PXQSATObjVariable, \PXQSATAuxVarOne, \PXQSATAuxVarPen, \PXQSATAuxVarTwo\\\PXQSATX{}_1, \PXQSATX{}_2, \PXQSATX{}_3, \PXQSATY{}
}}
\left\{
1-\PXQSATObjVariable
+
2 \PXQSATAuxVarPen{}_1
-
\sum_{i = 1}^{n_1 - 1} \PXQSATTruthAssignment{}_{1 i} / 2^{i}
+
6(\PXQSATY{} - 1/2) / 2^{n_1} :
\begin{array}{l}
\ZeroVector \le \PXQSATTruthAssignment{}_1, \PXQSATX{}_1 \le \OneVector, \\
(\PXQSATTruthAssignment{}_2, \PXQSATObjVariable, \PXQSATAuxVarOne, \PXQSATAuxVarPen, \PXQSATAuxVarTwo, \PXQSATX{}_2, \PXQSATX{}_3, \PXQSATY{})
\\
\ \ 
\in \OptSolSet\eqref{eq:LabelHardnessOfAttainSecondLevel}
\end{array}
\right\},
\tag{\EqtagHardnessOfAttain{3}{1}{(\PXQSATQBF)}}
\label{eq:LabelHardnessOfAttain}
\end{align}
\begin{align}
\inf_{\substack{
\PXQSATTruthAssignment{}_2, \PXQSATObjVariable, \PXQSATAuxVarOne, \PXQSATAuxVarPen, \PXQSATAuxVarTwo\\\PXQSATX{}_2, \PXQSATX{}_3, \PXQSATY{}
}}
\left\{
\PXQSATObjVariable
+ (\PXQSATCostCoefOne^{\PXQSATQBF})^{\top} \PXQSATAuxVarTwo
- \PXQSATX{}_2
:
\begin{array}{l}
\PXQSATX{}_2 \le \PXQSATX{}_1, \PXQSATX{}_3 = 0,
\ZeroVector \le \PXQSATTruthAssignment{}_2, \PXQSATObjVariable, \PXQSATAuxVarOne, \PXQSATX{}_2 \le \OneVector,
\\
(\PXQSATAuxVarPen, \PXQSATAuxVarTwo, \PXQSATX{}_3, \PXQSATY{})
\in \OptSolSet\eqref{eq:LabelHardnessOfAttainThirdLevel}
\end{array}
\right\},
\tag{\EqtagHardnessOfAttain{3}{2}{(\PXQSATTruthAssignment{}_1, \PXQSATX{}_1, \PXQSATQBF)}}
\label{eq:LabelHardnessOfAttainSecondLevel}
\end{align}
\begin{align}
\inf_{\substack{\PXQSATAuxVarPen, \PXQSATAuxVarTwo, \PXQSATX{}_3,\PXQSATY{} }}
\left\{
(\PXQSATCostCoefTwo^{\PXQSATQBF})^{\top} \PXQSATAuxVarTwo
-
\PXQSATX{}_3
:
\begin{array}{l}
\PXQSATQBFConstraintCoef^{\PXQSATQBF}_{1} \PXQSATTruthAssignment{}_1
+ \PXQSATQBFConstraintCoef^{\PXQSATQBF}_{2} \PXQSATTruthAssignment{}_2
+ \PXQSATQBFConstraintCoef^{\PXQSATQBF}_{3} \PXQSATObjVariable
+ \PXQSATQBFConstraintCoef^{\PXQSATQBF}_{4} \PXQSATAuxVarOne
+ \PXQSATQBFConstraintCoef^{\PXQSATQBF}_{5} \PXQSATAuxVarPen
+ \PXQSATQBFConstraintCoef^{\PXQSATQBF}_{6} \PXQSATAuxVarTwo \ge \PXQSATQBFConstraintRHS^{\PXQSATQBF},
\\
\PXQSATX{}_3 \le \PXQSATX{}_2,
\\
\PXQSATX{}_3 \le 1 - \PXQSATX{}_2,
\\
6 (\PXQSATY{} - 1/2) \ge
-\PXQSATTruthAssignment{}_{1 n_1},
\\
6 (\PXQSATY{} - 1/2) \ge
\PXQSATX{}_2 - \PXQSATX{}_1
-1 + \PXQSATObjVariable,
\\
\ZeroVector \le \PXQSATAuxVarPen, \PXQSATAuxVarTwo, \PXQSATX{}_3, \PXQSATY{} \le \OneVector
\end{array}
\right\}.
\tag{\EqtagHardnessOfAttain{3}{3}{(\PXQSATTruthAssignment{}_1, \PXQSATTruthAssignment{}_2, \PXQSATObjVariable, \PXQSATAuxVarOne, \PXQSATX{}_1, \PXQSATX{}_2, \PXQSATQBF)}}
\label{eq:LabelHardnessOfAttainThirdLevel}
\end{align}
Below, we show that it does not have an optimal solution if and only if $\PXQSATQBF$ \PXQSATIsTrue{} and the last component of $\PXQSATQBFLexMaxSolution(\PXQSATQBF)$ is $1$.

Variables corresponding to \eqref{eq:QLPSearch} and \EqrefExampleTwo{3}{1}{} are almost decoupled.
The only constraints joining them are
$6 (\PXQSATY{} - 1/2) \ge
-\PXQSATTruthAssignment{}_{1 n_1}$ and $
6 (\PXQSATY{} - 1/2) \ge
\PXQSATX{}_2 - \PXQSATX{}_1
-1 + \PXQSATObjVariable$.
However, these constraints are always inactive in the second and third players' problems (the coefficient $6$ is chosen sufficiently large to ensure this).
Thus, we can show that $\FeasSolSet\eqref{eq:LabelHardnessOfAttain} = \FeasSolSet\eqref{eq:HardnessOfAttainabilityDecoupled}$, where
\begin{align}
\inf_{\substack{
\PXQSATTruthAssignment{}_1, \PXQSATTruthAssignment{}_2, \PXQSATObjVariable, \PXQSATAuxVarOne, \PXQSATAuxVarPen, \PXQSATAuxVarTwo\\\PXQSATX{}_1, \PXQSATX{}_2, \PXQSATX{}_3, \PXQSATY{}
}} \ &
1 - \PXQSATObjVariable
+
2 \PXQSATAuxVarPen{}_1
-
\sum_{i = 1}^{n_1 - 1} \PXQSATTruthAssignment{}_{1 i} / 2^{i}
+
6(\PXQSATY{} - 1/2) / 2^{n_1}
\tag{\EqtagHardnessOfAttainDecoupled{3}{1}{(\PXQSATQBF)}}
\label{eq:HardnessOfAttainabilityDecoupled}
\\
\text{s.t.} \
&
6(\PXQSATY{} - 1/2) \ge - \PXQSATTruthAssignment{}_{1 \PXQSATNVariables{}_1},
\notag
\\
&
6(\PXQSATY{} - 1/2) \ge \PXQSATX{}_2 - \PXQSATX{}_1 - 1 + \PXQSATObjVariable,
\notag
\\
&
\ZeroVector \le \PXQSATTruthAssignment{}_1, \PXQSATX{}_1 \le \OneVector, 
\notag
\\
&
(\PXQSATTruthAssignment{}_2, \PXQSATObjVariable, \PXQSATAuxVarOne, \PXQSATAuxVarPen, \PXQSATAuxVarTwo)
\in \OptSolSet\EqrefNewQLPOriginal{3}{2}{(\PXQSATTruthAssignment{}_1, \PXQSATQBF)},
\notag
\\
&
(\PXQSATX{}_2, \PXQSATX{}_3)
\in \OptSolSet\eqref{LabelExampleTwoSecond}.
\notag
\end{align}
Therefore, we have $\OptValFunc\EqrefHardnessOfAttain{3}{1}{(\PXQSATQBF)} = \OptValFunc\eqref{eq:HardnessOfAttainabilityDecoupled} \allowbreak \ge \OptValFunc\EqrefNewQLPSearch{3}{1}{(\PXQSATQBF)}$ and $\OptSolSet\EqrefHardnessOfAttain{3}{1}{(\PXQSATQBF)} \allowbreak = \OptSolSet\eqref{eq:HardnessOfAttainabilityDecoupled}$.

\paragraph{\textbf{Case 1: $\PXQSATQBF$ \PXQSATIsNotTrue{}}}
By Lemma~\ref{lemma:theorem:search-obj-hardness-general}, we have $\OptValFunc\EqrefNewQLPSearch{3}{1}{(\PXQSATQBF)} = 1 / 2^{\PXQSATNVariables{}_1}$.
Let $(\PXQSATTruthAssignment{}_1, \PXQSATTruthAssignment{}_2, \PXQSATObjVariable, \PXQSATAuxVarOne, \PXQSATAuxVarPen, \PXQSATAuxVarTwo, \PXQSATX{}_1, \PXQSATX{}_2, \PXQSATX{}_3, \PXQSATY{})$ be such that
$(\PXQSATTruthAssignment{}_1, \PXQSATTruthAssignment{}_2, \PXQSATObjVariable, \PXQSATAuxVarOne, \PXQSATAuxVarPen, \PXQSATAuxVarTwo)\in \OptSolSet\EqrefNewQLPSearch{3}{1}{(\PXQSATQBF)}$,
$\PXQSATX{} = \ZeroVector$, and $6(\PXQSATY{} - 1/2) = -1$.
Using Lemma~\ref{lemma:theorem:search-obj-hardness-general} again, we have $\PXQSATTruthAssignment{}_1 = \OneVector$ and $\PXQSATObjVariable = 0$.
It follows that this solution is feasible for \eqref{eq:HardnessOfAttainabilityDecoupled}, and its objective value is $1 / 2^{\PXQSATNVariables{}_1}$, implying its optimality.

\paragraph{\textbf{Case 2: $\PXQSATQBF$ \PXQSATIsTrue{} and the last component of $\PXQSATQBFLexMaxSolution(\PXQSATQBF)$ is $0$}}
By Lemma~\ref{lemma:theorem:search-obj-hardness-general}, we have $\OptValFunc\EqrefNewQLPSearch{3}{1}{(\PXQSATQBF)} = - \sum_{i = 1}^{\PXQSATNVariables{}_1} \PXQSATQBFLexMaxSolution(\PXQSATQBF)_i / 2^{i}$.
Let $(\PXQSATTruthAssignment{}_1, \PXQSATTruthAssignment{}_2, \PXQSATObjVariable, \PXQSATAuxVarOne, \PXQSATAuxVarPen, \PXQSATAuxVarTwo, \PXQSATX{}_1, \PXQSATX{}_2, \PXQSATX{}_3, \PXQSATY{})$ be such that
$(\PXQSATTruthAssignment{}_1, \PXQSATTruthAssignment{}_2, \PXQSATObjVariable, \PXQSATAuxVarOne, \PXQSATAuxVarPen, \PXQSATAuxVarTwo)\in \OptSolSet\EqrefNewQLPSearch{3}{1}{(\PXQSATQBF)}$,
$\PXQSATX{} = \ZeroVector$, and $6(\PXQSATY{} - 1/2) = 0$.
Using Lemma~\ref{lemma:theorem:search-obj-hardness-general} again, we have $\PXQSATTruthAssignment{}_1 = \PXQSATQBFLexMaxSolution(\PXQSATQBF)$ and $\PXQSATObjVariable = 1$.
It follows that this solution is feasible for \eqref{eq:HardnessOfAttainabilityDecoupled}, and its objective value is $- \sum_{i = 1}^{\PXQSATNVariables{}_1} \PXQSATQBFLexMaxSolution(\PXQSATQBF)_i / 2^{i}$, implying its optimality.

\paragraph{\textbf{Case 3: $\PXQSATQBF$ \PXQSATIsTrue{} and the last component of $\PXQSATQBFLexMaxSolution(\PXQSATQBF)$ is $1$}}
By Lemma~\ref{lemma:theorem:search-obj-hardness-general}, we have $\OptValFunc\EqrefNewQLPSearch{3}{1}{(\PXQSATQBF)} = - \sum_{i = 1}^{\PXQSATNVariables{}_1} \PXQSATQBFLexMaxSolution(\PXQSATQBF)_i / 2^{i}$.
Let $\SequenceIndex$ be any positive integer and let $(\PXQSATTruthAssignment{}_1, \PXQSATTruthAssignment{}_2, \PXQSATObjVariable, \PXQSATAuxVarOne, \PXQSATAuxVarPen, \PXQSATAuxVarTwo, \PXQSATX{}_1, \PXQSATX{}_2, \PXQSATX{}_3, \PXQSATY{})$ be such that
$(\PXQSATTruthAssignment{}_1, \PXQSATTruthAssignment{}_2, \PXQSATObjVariable, \PXQSATAuxVarOne, \PXQSATAuxVarPen, \PXQSATAuxVarTwo)\in \OptSolSet\EqrefNewQLPSearch{3}{1}{(\PXQSATQBF)}$,
$\PXQSATX{}_1 = 1 - 1/\SequenceIndex$, $\PXQSATX{}_2 = \PXQSATX{}_3 = 0$, and $6(\PXQSATY{} - 1/2) = -(1 - 1/\SequenceIndex)$.
Using Lemma~\ref{lemma:theorem:search-obj-hardness-general} again, we have $\PXQSATTruthAssignment{}_1 = \PXQSATQBFLexMaxSolution(\PXQSATQBF)$ and $\PXQSATObjVariable = 1$.
It follows that this solution is feasible for \eqref{eq:HardnessOfAttainabilityDecoupled}, and its objective value is $- \sum_{i = 1}^{\PXQSATNVariables{}_1 - 1} \PXQSATQBFLexMaxSolution(\PXQSATQBF)_i / 2^{i} - (1 - 1/\SequenceIndex) / 2^{i}$.
Since $\SequenceIndex$ was an arbitrary positive integer, $\OptValFunc\eqref{eq:HardnessOfAttainabilityDecoupled} = - \sum_{i = 1}^{\PXQSATNVariables{}_1} \PXQSATQBFLexMaxSolution(\PXQSATQBF)_i / 2^{i}$.

However, in the last case, \eqref{eq:HardnessOfAttainabilityDecoupled} does not have an optimal solution.
Let $(\PXQSATTruthAssignment{}_1, \PXQSATTruthAssignment{}_2, \PXQSATObjVariable, \PXQSATAuxVarOne, \PXQSATAuxVarPen, \PXQSATAuxVarTwo, \PXQSATX{}_1, \PXQSATX{}_2, \PXQSATX{}_3, \PXQSATY{})$ be any feasible solution for \eqref{eq:HardnessOfAttainabilityDecoupled}.
If $\PXQSATTruthAssignment{}_1$ is not binary, we have $\PXQSATAuxVarPen{}_1 = 1$ (Lemma~\ref{lemma:ValueOfQLP}), implying that it is suboptimal.
If $\PXQSATTruthAssignment{}_1$ is binary but $\PXQSATQBF(\PXQSATTruthAssignment{}_1)$ \PXQSATIsNotTrue{}, we have $\PXQSATObjVariable = 0$, again implying its suboptimality.
If $\PXQSATTruthAssignment{}_1$ is binary, $\PXQSATQBF(\PXQSATTruthAssignment{}_1)$ \PXQSATIsTrue{} but $\PXQSATTruthAssignment{}_1 \ne \PXQSATQBFLexMaxSolution(\PXQSATQBF)$, since $\PXQSATObjVariable = 1$ and $\PXQSATAuxVarPen{}_1 = 0$, its objective value can be bounded from below as
\begin{align*}
1 - \PXQSATObjVariable
+
2 \PXQSATAuxVarPen{}_1
-
\sum_{i = 1}^{n_1 - 1} \PXQSATTruthAssignment{}_{1 i} / 2^{i}
+
6(\PXQSATY{} - 1/2) / 2^{n_1}
&\ge
- \sum_{i = 1}^{\PXQSATNVariables{}_1} \PXQSATTruthAssignment{}_{1 i} / 2^{i}
\\
&>
- \sum_{i = 1}^{\PXQSATNVariables{}_1} \PXQSATQBFLexMaxSolution(\PXQSATQBF)_i / 2^{i}
\\
&=
\OptValFunc\eqref{eq:HardnessOfAttainabilityDecoupled}.
\end{align*}
If $\PXQSATTruthAssignment{}_1$ is binary and $\PXQSATTruthAssignment{}_1 = \PXQSATQBFLexMaxSolution(\PXQSATQBF)$, we have $\PXQSATObjVariable = 1$, $\PXQSATAuxVarPen{}_1 = 0$ and $\PXQSATX{}_2 - \PXQSATX{}_1 > -1$, thus
\begin{align*}
1 - \PXQSATObjVariable
+
2 \PXQSATAuxVarPen{}_1
-
\sum_{i = 1}^{n_1 - 1} \PXQSATTruthAssignment{}_{1 i} / 2^{i}
+
6(\PXQSATY{} - 1/2) / 2^{n_1}
&\ge
- \sum_{i = 1}^{\PXQSATNVariables{}_1 - 1} \PXQSATTruthAssignment{}_{1 i} / 2^{i}
+
(\PXQSATX{}_2 - \PXQSATX{}_1) / 2^{n_1}
\\
&>
- \sum_{i = 1}^{\PXQSATNVariables{}_1} \PXQSATQBFLexMaxSolution(\PXQSATQBF)_i / 2^{i}
\\
&=
\OptValFunc\eqref{eq:HardnessOfAttainabilityDecoupled}.
\end{align*}
Thus, there is no feasible solution that attains the optimal objective value.

Instance \EqrefHardnessOfAttain{3}{1}{(\PXQSATQBF)} has fractional coefficients of exponentially small magnitude, but one can obtain an equivalent rational \DecisionProblemModifierBoxedPoly{\KLPInstances{3}} instance
as in the proof of Theorem~\ref{theorem:search-obj-hardness-bilevel}.
The resulting instance can be constructed in logarithmic space, implying $\ClassDeltaP{3}$-hardness of \DecisionProblemModifierCompBoxedPoly{\DecisionProblemAttain{3}}.
This implies \DecisionProblemModifierBoxedPoly{\DecisionProblemAttain{3}} is $\ClassDeltaP{3}$-hard, for $\ClassDeltaP{3}$ is closed under complementation.

The proof for assertion~(ii) is similar.
\end{proof}

\subsection{Feasibility Problems}\label{sec:HardnessProofs:FeasibilityProblems}

The goal of this section is to establish the following hardness results in $k$-level LP.

\begin{theorem}
\label{theorem:klp-hardness}
The following statements hold.
\begin{enumerate}
\item
For $k \ge 2$, the decision problem \DecisionProblemModifierBoxedPoly{\DecisionProblemFeas{k}} is \ClassSigmaP{k - 1}-hard.
\item
For $k \ge 3$, the decision problem \DecisionProblemModifierNoLinkPoly{\DecisionProblemFeas{k}} is \ClassPiP{k - 2}-hard.
\item
For $k \ge 5$, the decision problem \DecisionProblemModifierNoLinkBoxedPoly{\DecisionProblemFeas{k}} is \ClassSigmaP{k - 1}-hard.
\end{enumerate}
\end{theorem}

\begin{proof}[Proof of Theorem~\ref{theorem:klp-hardness}]
\begin{enumerate}
\item
We only show $k = 2$ since the proof for the general case is analogous.
Consider a \DecisionProblemModifierBoxedPoly{\KLPInstances{2}} instance obtained by adding a linking constraint ``objective $\le 0$'' to \EqrefSATBLP{2}{1}{(\PXSATBooleanFormula)}.
The resulting instance is feasible if and only if $\PXSATBooleanFormula$ is satisfiable, implying the claim.

\item
Let $k \ge 2$ and $\PXQSATQBF \in \PXQSATQBFSet{k-1}$.
Let $(A, b, c)$ be a \DecisionProblemModifierNoLinkPoly{\KLPInstances{k}} instance whose optimal objective value is $-1$ if $\PXQSATQBF$ \PXQSATIsTrue{} and $0$ otherwise.
Such an instance can be obtained by modifying \EqrefNewQLPOriginal{k}{1}{(\PXQSATQBF)}.
Consider
\begin{equation}
\label{eq:feasibility-k-lp-without-linking}
\inf_{x', s'} \left\{ 0 :
\begin{pmatrix}
x' \\ s'
\end{pmatrix}
\in
\Argmin_{x, s \ge 1} \left\{
\sum_{j = 1}^{k} c^{\top} x_{j}
:
\begin{pmatrix}
x_1 \\ \vdots \\ x_{k}
\end{pmatrix}
\in \FeasSolSet\EqrefGeneralKLP{k}{1}{(A, s b, c)} \right\}
\right\}.
\end{equation}
Instance~\eqref{eq:feasibility-k-lp-without-linking} is a rational \DecisionProblemModifierNoLinkPoly{\KLPInstances{(k+1)}} instance.
In light of Lemma~\ref{lemma:homogeneous-klp-without-linking-constraints}, it is feasible if and only if $\PXQSATQBF \not\in \PXQSATQBFSet{k-1}$.
Thus, \DecisionProblemModifierNoLinkPoly{\DecisionProblemFeas{k}} is \ClassPiP{k - 2}-hard for $k \ge 3$.

\item
We prove for the case $k=5$.
Extension to $k \ge 6$ is straightforward.

Let $\PXQSATQBF \in \PXQSATQBFSet{4}$.
Consider the following instance:
\begin{align}
\inf_{\substack{
\PXQSATTruthAssignment{}_1,
\PXQSATTruthAssignment{}_2, \PXQSATTruthAssignment{}_3, \PXQSATTruthAssignment{}_4, \PXQSATObjVariable, \PXQSATAuxVarOne, \PXQSATAuxVarPen, \PXQSATAuxVarTwo\\\PXQSATX{}_1, \PXQSATX{}_2, \PXQSATX{}_3, \PXQSATX{}_4, \PXQSATY{},\\
\PXQSATTruthAssignment{}_2', \PXQSATTruthAssignment{}_3', \PXQSATTruthAssignment{}_4', \PXQSATObjVariable', \PXQSATAuxVarOne', \PXQSATAuxVarPen', \PXQSATAuxVarTwo'\\\PXQSATX{}_1', \PXQSATX{}_2', \PXQSATX{}_3', \PXQSATX{}_4', \PXQSATY{}'
}}
\left\{
0 :
\begin{array}{l}
\ZeroVector \le \PXQSATTruthAssignment{}_1 \le \OneVector,
\\
(\PXQSATTruthAssignment{}_2, \PXQSATTruthAssignment{}_3, \PXQSATTruthAssignment{}_4, \PXQSATObjVariable, \PXQSATAuxVarOne, \PXQSATAuxVarPen, \PXQSATAuxVarTwo, \PXQSATX{}_1, \PXQSATX{}_2, \PXQSATX{}_3, \PXQSATX{}_4, \PXQSATY{})
\\
\qquad\qquad\qquad\qquad\qquad\qquad
\in \OptSolSet\eqref{eq:HardFeasDecoupled},
\\
(\PXQSATTruthAssignment{}_2', \PXQSATTruthAssignment{}_3', \PXQSATTruthAssignment{}_4', \PXQSATObjVariable', \PXQSATAuxVarOne', \PXQSATAuxVarPen', \PXQSATAuxVarTwo', \PXQSATX{}_1', \PXQSATX{}_2', \PXQSATX{}_3', \PXQSATX{}_4', \PXQSATY{}')
\\
\qquad\qquad\qquad\qquad\qquad\qquad
\in \OptSolSet\eqref{eq:HardFeasDecoupledTwo}
\end{array}
\right\},
\tag{\EqtagHardFeas{5}{1}{(\PXQSATQBF)}}
\label{eq:HardFeas}
\end{align}
where
\begin{align}
\inf_{\substack{\PXQSATTruthAssignment{}_2, \PXQSATTruthAssignment{}_3, \PXQSATTruthAssignment{}_4, \PXQSATObjVariable, \PXQSATAuxVarOne, \PXQSATAuxVarPen, \PXQSATAuxVarTwo\\\PXQSATX{}_1, \PXQSATX{}_2, \PXQSATX{}_3, \PXQSATX{}_4, \PXQSATY{}}}
\ &
6 (\PXQSATY{} - 1/2)
\label{eq:HardFeasDecoupled}
\tag{\EqtagHardFeasDecoupled{5}{2}{(\PXQSATTruthAssignment{}_1, \PXQSATQBF)}}
\\
\text{s.t.} \ &
6 (\PXQSATY{} - 1/2) \ge -1,
\notag
\\
&
6 (\PXQSATY{} - 1/2) \ge \PXQSATX{}_3-\PXQSATX{}_1 + \PXQSATAuxVarPen{}_1 - 1,
\notag
\\
&
(\PXQSATTruthAssignment{}_3, \PXQSATTruthAssignment{}_4, \PXQSATObjVariable, \PXQSATAuxVarOne, \PXQSATAuxVarPen, \PXQSATAuxVarTwo)
\in \OptSolSet\EqrefNewQLPOriginal{5}{3}{(\PXQSATTruthAssignment{}_1, \PXQSATTruthAssignment{}_2, \PXQSATQBF)},
\notag
\\
&
(\PXQSATX{}_2, \PXQSATX{}_3, \PXQSATX{}_4)
\in \OptSolSet\eqref{eq:LabelExampleOneSecondPlayersProblem},
\notag
\end{align}
and
\begin{align}
\inf_{\substack{\PXQSATTruthAssignment{}_2, \PXQSATTruthAssignment{}_3, \PXQSATTruthAssignment{}_4, \PXQSATObjVariable, \PXQSATAuxVarOne, \PXQSATAuxVarPen, \PXQSATAuxVarTwo\\\PXQSATX{}_1, \PXQSATX{}_2, \PXQSATX{}_3, \PXQSATX{}_4, \PXQSATY{}}}
\ &
6 (\PXQSATY{} - 1/2)
\label{eq:HardFeasDecoupledTwo}
\tag{\EqtagHardFeasDecoupledTwo{5}{2}{(\PXQSATTruthAssignment{}_1, \PXQSATQBF)}}
\\
\text{s.t.} \ &
6 (\PXQSATY{} - 1/2) \ge \PXQSATObjVariable + 2 \PXQSATAuxVarPen{}_2 - 1,
\notag
\\
&
6 (\PXQSATY{} - 1/2) \ge \PXQSATX{}_3-\PXQSATX{}_1,
\notag
\\
&
(\PXQSATTruthAssignment{}_3, \PXQSATTruthAssignment{}_4, \PXQSATObjVariable, \PXQSATAuxVarOne, \PXQSATAuxVarPen, \PXQSATAuxVarTwo)
\in \OptSolSet\EqrefNewQLPOriginal{5}{3}{(\PXQSATTruthAssignment{}_1, \PXQSATTruthAssignment{}_2, \PXQSATQBF)},
\notag
\\
&
(\PXQSATX{}_2, \PXQSATX{}_3, \PXQSATX{}_4)
\in \OptSolSet\eqref{eq:LabelExampleOneSecondPlayersProblem}.
\notag
\end{align}
As in the proof of Theorem~\ref{theorem:hardness:attainability}, we can rewrite \EqrefHardFeas{5}{2}{(\PXQSATQBF)} as a rational \DecisionProblemModifierNoLinkBoxedPoly{\KLPInstances{5}} instance.
Similar arguments to the proof of Theorem~\ref{theorem:hardness:attainability} can show:
\begin{itemize}
\item
If $\PXQSATTruthAssignment{}_1$ is not binary, $\OptSolSet\eqref{eq:HardFeasDecoupled}$ is empty;
\item
If $\PXQSATTruthAssignment{}_1 \in \{0, 1\}^{n_1}$ and $\PXQSATQBF(\PXQSATTruthAssignment{}_1)$ \PXQSATIsNotTrue{}, $\OptSolSet\eqref{eq:HardFeasDecoupledTwo}$ is empty;
\item
If $\PXQSATTruthAssignment{}_1 \in \{0, 1\}^{n_1}$ and $\PXQSATQBF(\PXQSATTruthAssignment{}_1)$ \PXQSATIsTrue{}, $\OptSolSet\eqref{eq:HardFeasDecoupled}$ and $\OptSolSet\eqref{eq:HardFeasDecoupledTwo}$ are nonempty.
\end{itemize}

Therefore, \eqref{eq:HardFeas} is feasible if and only if $\PXQSATQBF$ \PXQSATIsTrue{}.
Thus, \DecisionProblemQThreeSAT{4} $\le_l$ \DecisionProblemModifierNoLinkBoxedPoly{\DecisionProblemFeas{5}}, implying its \ClassSigmaP{4}-hardness.
\MyQED
\end{enumerate}
\end{proof}

\section{Extensions and Limitations}\label{sec:extensions_and_limitations}

In this section, we discuss several limitations of our results, as well as possible extensions.

First, we showed that deciding unboundedness of a $4$-level LP instance is \ClassPiP{2}-hard and belongs to \ClassSigmaP{3}.
Whether it is complete for either of these complexity classes remains open. A closer inspection of the proof of Theorem~\ref{theorem:membership} suggests that extending Lemma~\ref{lemma:EquivalentConditionsForFeasibility} to \KLPInstances{4} would imply \ClassPiP{2}-completeness. Such an extension, in turn, would follow from a generalization of the sensitivity analysis developed in Section~\ref{sec:sensitivity-of-bilevel-linear-program-without-linking-constraints}, in particular, Theorem~\ref{theorem:lipschitz-continuity-of-blp-value-function}, to trilevel LP. However, establishing these properties appears to require a deeper understanding of the structure and properties of trilevel LPs.

Second, this paper focuses on the search problem of computing the optimal objective value. A closely related problem is that of computing an optimal solution. A natural approach is as follows: first compute the optimal objective value. Then, iteratively refine the solution by adding a constraint enforcing that the objective value equals the computed optimum, and redefining the objective to minimize a selected component of $x_1$ (e.g., the first component). By resolving the modified problem, one can determine the value of that component in some optimal solution (or, in the case where the optimum is not attained, after applying a suitable perturbation). Repeating this procedure for each component, fixing values one by one, yields a complete optimal solution.

This algorithm terminates in finitely many steps and correctly computes an optimal solution, when one exists. Moreover, it runs in polynomial time if the number $\NVariables_1$ of the first player's variables is fixed. However, it is unclear whether the algorithm runs in time polynomial in  $\NVariables_1$ or in the encoding size of the input. 
In particular, each iteration (i.e., fixing a variable) introduces constraints with increasingly large coefficients, potentially causing superpolynomial growth in the encoding size.


Third, our computational complexity results depend on how the $k$-level problem is formulated. 
Following \MyCitet{Jeroslow}{Jeroslow1985}, in our formulation, each player selects an optimal solution whenever one exists.
If, for example, the second player's optimum is unattainable, the corresponding first player's decision is considered infeasible.
Another common formulation is based on $\epsilon$-optimal solution sets, where followers (i.e., the second through $k$-th players) are indifferent among feasible solutions whose suboptimality is below a prescribed threshold. 
Under this setting, deciding attainability under conditions~\ref{Condition1}--\ref{Condition3} becomes polynomial-time solvable. 
A more comprehensive study of alternative formulations remains an interesting direction for future research.

Finally, on the positive side, our arguments extend naturally to multilevel mixed-binary LP.
A $k$-level mixed-binary LP can be transformed into a $(k+1)$-level LP by introducing an additional player who evaluates the fractionality of variables and linking constraints enforcing binary constraints~\MyCite{Sugishita2026}.
For example, this shows that the search problem for a $k$-level mixed-binary LP is in \ClassFDeltaP{k+1}.
Hardness follows by modifying the argument in Section~\ref{sec:hardnessproofs:searchproblems:kgethree}, together with the pure-binary instance of \MyCitet{Jeroslow}{Jeroslow1985}.
A summary of these results is provided in Table~\ref{tab:mixedbinary}.
Since the proofs are either straightforward or closely parallel those for multilevel LP, they are omitted.
These results strengthen that of \MyCitet{Rodrigues \EtAl{}}{RodriguesEtAl2024}, who show that deciding the unboundedness of a $k$-level mixed-binary LP with linking constraints is \ClassSigmaP{k}-hard.
As in multilevel LP, the computational complexity of \DecisionProblemFeas{k} exhibits a jump from $k = 2$ to $k \ge 3$ for mixed-binary multilevel LPs without linking constraints, assuming the polynomial hierarchy does not collapse.
We note that special care is required when studying the pure-binary case.
It is straightforward to show that the search problem for a $k$-level pure-binary LP is \ClassFDeltaP{k+1}-complete.
However, when the input instances are restricted to have coefficients of polynomial magnitude (condition~\ref{Condition3}), the search problem belongs to \ClassFThetaP{k+1} (also denoted as $\text{FP}^{\Sigma^p_{k}[\log n]}$), since the number of iterations required by binary search is logarithmically bounded.

\begin{table}
\centering
\caption{Computational complexity of decision and search problems for mixed-binary multilevel LPs. All results remain valid even when the input instances are restricted to have coefficients of polynomial magnitude.}
\label{tab:mixedbinary}
\footnotesize
\setlength\extrarowheight{0.2em}
\setlength{\tabcolsep}{2pt}
\begin{NiceTabular}{ccccccc}
\CodeBefore
\Body
\toprule
& & & \multicolumn{2}{c}{With Linking Constraints} & \multicolumn{2}{c}{Without Linking Constraints} \\
\cmidrule(lr){4-5}
\cmidrule(lr){6-7}
& \multicolumn{1}{c}{$k$} & & \multicolumn{1}{c}{Unbounded $x$} & \multicolumn{1}{c}{Bounded $x$} & \multicolumn{1}{c}{Unbounded $x$} & \multicolumn{1}{c}{Bounded $x$} \\
\midrule
\DecisionProblemVal{k}
& $\ge 2$ & 
\cite{Sugishita2026}
& \ClassSigmaP{k}-complete & \ClassSigmaP{k}-complete & \ClassSigmaP{k}-complete & \ClassSigmaP{k}-complete \\
\midrule
\multirow[m]{2}{*}{\DecisionProblemUnb{k}}
& $2$ & 
& 
\ClassSigmaP{2}-complete
& - & 
\ClassNP{}-hard
& - \\
& $\ge 3$ && \ClassSigmaP{k}-complete & - & \ClassSigmaP{k}-complete & - \\
\midrule
\multirow[m]{2}{*}{\DecisionProblemFeas{k}}
& $2$ && \ClassSigmaP{2}-complete & \ClassSigmaP{2}-complete & 
\ClassNP{}-complete
& 
\ClassNP{}-complete
\\
& $\ge 3$ && \ClassSigmaP{k}-complete & \ClassSigmaP{k}-complete & \ClassSigmaP{k}-complete & \ClassSigmaP{k}-complete \\
\midrule
\multirow[m]{1}{*}{\DecisionProblemAttain{k}}
& $\ge2$ & & $\ClassDeltaP{k + 1}$-complete & $\ClassDeltaP{k + 1}$-complete & $\ClassDeltaP{k + 1}$-complete & $\ClassDeltaP{k + 1}$-complete \\
\midrule
\SearchProblemObj{k}
& $\ge 2$ & & \ClassFDeltaP{k + 1}-complete & \ClassFDeltaP{k + 1}-complete & \ClassFDeltaP{k + 1}-complete & \ClassFDeltaP{k + 1}-complete \\
\bottomrule
\CodeAfter
\end{NiceTabular}
\end{table}

\section{Conclusions}\label{sec:Conclusions}

In this paper, we have presented a comprehensive analysis of the computational complexity of multilevel LP, focusing on the fundamental problems of feasibility, existence of optimal solutions, and computation of optimal objective values. Our results provide a unified view of how these problems are situated within the polynomial hierarchy.

A central insight of our study is that the computational complexity of multilevel LP is highly sensitive to both the number of levels~$k$ and the structural properties of the model, such as the presence of linking constraints and unbounded variables.
In particular, we established that the feasibility problem is $\Sigma^p_{k-1}$-complete for all $k \ge 2$ in the general case, while exhibiting polynomial-time solvability under restrictive assumptions when $k$ is small.
The sharp transition in complexity (assuming the polynomial hierarchy does not collapse) between $k=4$ and $k=5$ in the absence of linking constraints and unbounded variables highlights a threshold phenomenon that, to the best of our knowledge, has not been previously identified in multilevel optimization.
We observed a similar phenomenon in the attainability decision problem for multilevel LP.

In addition, we showed that computing the optimal objective value is $\mathrm{F}\Delta^p_k$-complete for all $k \ge 2$, even under simplifying assumptions. This result suggests that, from a computational perspective, the search version of multilevel LP remains intractable across all $k \ge 2$, and does not admit the same types of tractability results observed for certain decision problems in restricted settings.

Beyond these complexity classifications, our results yield several structural implications. 
Notably, they establish limitations on polynomial-time transformations between formulations with and without linking constraints, unless widely believed complexity-theoretic assumptions fail. 
At the same time, they identify regimes (in particular, for sufficiently large~$k$) where such transformations become possible for feasibility.
However, they do not preserve optimal objective values, as they concern only the associated decision problems.
These findings contribute to a more nuanced understanding of the modeling trade-offs inherent in multilevel LP.

\appendix

\AppendixHeader

\bibliographystyle{siamplain}
\bibliography{references}

\end{document}